\documentclass[10pt]{amsart}

\usepackage{amssymb}
\usepackage{amsthm}
\usepackage{dsfont}
\usepackage{tikz-cd}

\theoremstyle{definition}

\newtheorem{theorem}{Theorem}[section]

\newtheorem{corollary}[theorem]{Corollary}
\newtheorem{proposition}[theorem]{Proposition}

\newtheorem{example}[theorem]{Example}

\newtheorem{acknowledgments}[theorem]{Acknowledgments}

\DeclareMathOperator{\Con}{Con}
\DeclareMathOperator{\Aut}{Aut}
\DeclareMathOperator{\Hom}{Hom}

\DeclareMathOperator{\id}{id}
\DeclareMathOperator{\im}{im}
\DeclareMathOperator{\ar}{ar}
\DeclareMathOperator{\Id}{Id}

\begin{document}

\title{Extensions realizing affine datum : the Wells derivation}
\author{Alexander Wires}
\address{School of Mathematical Sciences, University of Electronic Science and Technology of China, Chengdu, 611731, Sichuan, PRC }
\email{awires81@uestc.edu.cn}

\begin{abstract}
We develop the Wells derivation for extensions realizing affine datum in arbitrary varieties; in particular, we show there is an exact sequence connecting the group of compatible automorphisms determined by the datum and the subgroup of automorphisms of an extension which preserves the extension's kernel. This implies a homomorphism between $2^{\mathrm{nd}}$-cohomology groups which realizes a group of kernel-preserving automorphisms of an extension as itself an extension of a subgroup of compatible automorphisms by the group of derivations of the datum. A refinement of this general Wells's-type theorem is given for a restricted class of varieties with a difference term which include any variety of groups with multiple operators in the sense of Higgins. The same results are obtained for nonabelian extensions in any variety of $R$-modules expanded by multilinear operations.
\end{abstract}


\maketitle


\section{Introduction}
\vspace{0.3cm}

We consider an application of the machinery developed in Wires \cite{wires1} for extensions which realize affine datum in arbitrary varieties of universal algebras. Our reference is the following theorem of Charles Wells:

\begin{theorem}\label{thm:1}(\cite{wells})
Let $\pi: G \rightarrow Q$ be a surjective group homomorphism with $K = \ker \pi$. There are homomorphisms and a set map $C(Q,K,\alpha) \rightarrow H_{\bar{\alpha}}^{2}(Q,ZK)$ such that
\[
1 \longrightarrow Z_{\bar{\alpha}}^{1}(Q,ZK) \longrightarrow \Aut_{K} G \longrightarrow C(Q,K,\alpha) \longrightarrow H_{\bar{\alpha}}^{2}(Q,ZK)
\]
is exact.
\end{theorem}

The extension $\pi: G \rightarrow Q$ determines datum $(Q,K,\alpha)$ where the homomorphism $\alpha:Q \rightarrow \mathrm{Out} \, K$ is induced by the conjugation action of $Q$ on $K$ afforded by a lifting of $\pi$. The group $\Aut_{K} G$ is the subgroup of automorphisms of $G$ which preserve $K$ set-wise. There is an action of $\Aut K \times \Aut Q$ on the 2-cochains of the datum given by simultaneously permuting their domains and codomains in the natural way. The subgroup $C(Q,K,\phi) \leq \Aut K \times \Aut Q$ consists of those pairs of automorphisms which satisfy a compatibility condition concerning the homomorphism $\alpha$ which guarantees that the set of 2-cocycles is closed under the action of the subgroup $C(Q,K,\alpha)$. The map $C(Q,K,\alpha) \rightarrow H_{\bar{\alpha}}^{2}(Q,ZK)$ is commonly referred to as the \emph{Wells map}, or \emph{Wells derivation}, and is indeed a principal derivation for the action of compatible automorphisms $C(Q,K,\alpha)$ on $H_{\bar{\alpha}}^{2}(Q,ZK)$. The proof of Wells's theorem is intimately related to how the equivalence on 2-cocycles which defines second-cohomology classes is characterized by stabilizing isomorphisms between extensions. A more thorough explanation can be found in the monograph of Passi, Singh and Yadav \cite{autgrp} and in Wells \cite{wells}, of course.

It is important to note that Wells's theorem applies to all group extensions and not just those with abelian kernels; however, we can utilize the development in Wires \cite{wires1} to consider Wells's argument in the restricted case of group extensions with abelian kernels and extend it to the general setting of extensions realizing affine datum in arbitrary varieties of universal algebras. This is the principal content of this manuscript.

\begin{theorem}\label{thm:2}
Suppose $\mathcal U$ is a variety containing affine datum $(Q,A^{\alpha,\tau},\ast)$ which is realized by an extension $\pi: A \rightarrow Q$ with associated 2-cocycle $T$. Then we have the exact sequence
\begin{align*}
 1 \longrightarrow \mathrm{Der}(Q,A^{\alpha,\tau},\ast)  \longrightarrow \Aut_{\alpha} A \stackrel{\psi}{\longrightarrow} C(Q,A^{\alpha,\tau},\ast) \stackrel{W_{T}}{\longrightarrow} H^{2}_{\mathcal U}(Q,A^{\alpha,\tau},\ast).
\end{align*}
\end{theorem}

For group datum $(Q,K,\phi)$ with $K$ abelian, we write $K \rtimes_{(\phi,f)} Q$ for the extension realizing the datum defined by modifying the operations of the semidirect product $K \rtimes_{\phi} Q$ by the addition of the 2-cocycle $f$.

\begin{theorem}\label{thm:3}
Let $(Q,A^{\alpha,\tau},\ast)$ be affine datum contained in the variety $\mathcal U$. For each $[T] \in H^{2}_{\mathcal U}(Q,A^{\alpha,\tau},\ast)$ there is a group action $\phi_{T}: \ker W_{T} \rightarrow \Aut \mathrm{Der}(Q,A^{\alpha,\tau},\ast)$ which realizes the semidirect product 
\[
\Aut_{\hat{\alpha}} A_{T}(Q,A^{\alpha,\tau},\ast) \approx \mathrm{Der}(Q,A^{\alpha,\tau},\ast) \rtimes_{\phi_{T}} \ker W_{T} ;
\] 
in particular, $\Aut_{\hat{\alpha}} A(\alpha)/\Delta_{\alpha \alpha} \approx \mathrm{Der}(Q,A^{\alpha,\tau},\ast) \rtimes_{\phi_{T}} C(Q,A^{\alpha,\tau},\ast)$.
\end{theorem}

A difference term $m$ for a variety $\mathcal V$ is \emph{weakly-associative} if $\mathcal V \vDash x = m(m(x,y,z),z,y)$. By restricting to varieties with a weakly-associative difference term, the representation of extensions with abelian kernels and idempotent elements takes a form more closely analogous to the classic group case \cite[Sec 4]{wires1}. This provides a further decomposition of the compatible automorphisms into certain pairs of automorphisms of the quotient algebra and a privileged congruence class (Proposition~\ref{prop:diffclass}). We then use the refined representation to translate Theorem~\ref{thm:2} for these types of algebras.

\begin{corollary}\label{cor:1}
Let $A \in \mathcal V$ a variety with a weakly-associative difference term $m$ and $u \in A$ is a characteristic idempotent. If $\alpha \in \Con A$ is abelian and affords a representation $A \approx I_{\alpha} \rtimes_{\ast,T} Q$, then there is an exact sequence
\begin{align}\label{eqn:idemwells}
1 \longrightarrow \mathrm{Der}(Q,I_{\alpha},\ast)  \longrightarrow \Aut_{I_{\alpha}} A \stackrel{\psi}{\longrightarrow} C(Q,I_{\alpha},\ast) \stackrel{W_{T}}{\longrightarrow} H^{2}_{\mathcal V}(Q,I_{\alpha},\ast).
\end{align}
\end{corollary}

The hypothesis of the previous corollary is quite restrictive for general varieties with a difference term; nevertheless, it is satisfied for any variety of groups with multiple operators in the sense of Higgins \cite{higgins}.

The paper Wires \cite{wires1} examines the deconstruction/reconstruction of extensions realizing affine datum in arbitrary varieties of universal algebras. As a step toward general extensions (what is often called non-abelian cohomology), the paper Wires \cite{multiext} also examines the parameters for characterizing extensions in varieties of $R$-modules expanded by multilinear operations. For abelian ideals with unary actions, this agrees with the development for general affine datum. Varieties of $R$-modules expanded by multilinear operations can be seen as a special case of Higgin's groups with multiple operators formalisms \cite{higgins}, but are still general enough to include many examples of multilinear algebras such as rings (Everett \cite{rings}), associative algebras (Agore and Militaru \cite{agore}, Hochschild \cite{hoch}), Lie algebras (Inassaridze, Khmaladze and Ladra \cite{lie}), Liebniz algebras (Casas, Khmaladze and Ladra \cite{leibniz}), dendriform algebras and bilinear Rota-Baxter algebras (Das and Rathee \cite{das}), Lie-Yamaguti algebras (Yamaguti \cite{yamaguti}) or conformal algebras (Bakalov, Kac and Voronov\cite{conform1}, Hou and Zhao \cite{conform2}, Smith \cite{smith1}) to name just a few well-studied classes. The constructions in \cite{wires1} parametrizing nonabelian extensions recovers in a uniform manner the cohomological classification of extensions previously developed for these different varieties. As examples of groups with multiple operators, Corollary~\ref{cor:1} is then the analogue of Wells's Theorem for extensions with abelian ideals. The nonabelian version of Wells's Theorem for these varieties proceeds in the same manner as the affine datum case with a modification; namely, the values of the Wells map resides in the free abelian group generated by $2^{\mathrm{nd}}$-cohomology for the variety. This is done to provide the Wells map with an appropriate codomain since cohomology classes in $H^{2}_{\mathcal V}(Q,I)$ for an arbitrary subvariety $\mathcal V$ may not be closed under the natural addition of 2-cocycles induced by $I$.

\begin{theorem}\label{thm:4}
Let $\mathcal V$ be a variety of $R$-modules expanded by multilinear operations. If $A \in \mathcal V$ is an extension $\pi: A \rightarrow Q$ with $I = \ker \pi$ and has associated 2-cocycle $T$, then there exists an exact sequence
\[
1 \longrightarrow \mathrm{Der}(Q,I) \longrightarrow \Aut_{I} A \stackrel{\psi}{\longrightarrow} \Aut I \times \Aut Q \stackrel{W_{T}}{\longrightarrow} FA \left( H^{2}_{\mathcal V}(Q,I) \right).
\]
\end{theorem}

\vspace{0.5cm}


\section{Preliminaries}\label{section:2}
\vspace{0.3cm}

In this section, we define the maps and the groups which appear in Theorem~\ref{thm:2}. For a tuple $\vec{a}=(a_1,\ldots,a_n) \in A^{n}$, the partial tuples determined by $1 \leq i < n$ are denoted by $\vec{a}_{i} = (a_1,\ldots,a_i)$ and $_{i}\vec{a} = (a_{i+1},\ldots,a_{n})$. For a map $f:A \rightarrow B$ and $\vec{a} \in A^{n}$, we write $f(\vec{a}) = (f(a_{1}),\ldots,f(a_{n})) \in B^{n}$ for the tuple of coordinate-wise evaluations.

Fix affine datum $(Q,A^{\alpha,\tau},\ast)$ in the signature $\tau$. According to Wires \cite{wires1}, there is a unique semidirect product $\rho: A(\alpha)/\Delta_{\alpha \alpha} \rightarrow Q$ realizing the datum which can be reconstructed from the operations $F_{f}$ for each signature symbol $f \in \tau$ which are given by
\[
F_{f}(\vec{a}) = f^{\Delta} \left( a_{1},\delta(l \circ \rho(a_{2})),\ldots,\delta(l \circ \rho(a_{n})) \right) +_{u} \sum_{i=2}^{n} a(f,i) \left(  \rho(a_{1}),\ldots,\rho(a_{i-1}),a_{i},\rho(a_{i+1}),\ldots,\rho(a_{n}) \right)  
\]
for $\vec{a} \in A(\alpha)/\Delta_{\alpha \alpha}$ where $n=\ar f$ and $u=l \circ f^{Q}(\rho(\vec{a}))$ for any choice of lifting $l:Q \rightarrow A$ associated to the datum. The operation $+_{u}$ is given by $x +_{u} y = m(x,\delta(u),y)$ for the ternary operation $m$ prescribed by the datum. The map $\delta : A \rightarrow A(\alpha)/\Delta_{\alpha \alpha}$ is the diagonal embedding $\delta(a) = \begin{bmatrix} a \\ a \end{bmatrix}/\Delta_{\alpha \alpha}$. Note for unary and nullary symbols $f \in \tau$, the operations $F_{f}$ are interpreted using only the partial operations $f^{\Delta}$ of the partial structure $A^{\alpha,\tau}$. According to Wires \cite{wires1}, the extensions in a variety $\mathcal U$ which realize the datum can be uniquely reconstructed by the algebras $A_{T}(Q,A^{\alpha,\tau},\ast)$ parametrized by the $\mathcal U$-compatible 2-cocycles $T$. The operations of the algebra are defined by adding the parameter $T$ to the operations of the semidirect product according to 
\begin{align*}
\begin{split}
F_{f}(\vec{a}) &= f^{\Delta} \left( a_{1},\delta(l \circ \rho(a_{2})),\ldots,\delta(l \circ \rho(a_{n})) \right) +_{u} \sum_{i=2}^{n} a(f,i) \left(  \rho(a_{1}),\ldots,\rho(a_{i-1}),a_{i},\rho(a_{i+1},\ldots,\rho(a_{n}) \right) \\
&\quad +_{u} T_{f} \left( \rho(\vec{a}) \right).
\end{split}
\end{align*}
Taken together, the $\mathcal U$-compatible 2-cocycles form the $2^{\mathrm{nd}}$-cohomology group $H^{2}_{\mathcal U}(Q,A^{\alpha,\tau},\ast)$ parametrizing as an abelian group the extensions in $\mathcal U$ realizing the datum. Details can be found in Wires \cite{wires1}.

The \emph{compatible automorphisms} for the datum is the set $C(Q,A^{\alpha,\tau},\ast)$ which consists of all pairs $(\sigma,\kappa) \in \Aut A(\alpha)/\Delta_{\alpha \alpha} \times \Aut Q$ such that for each $f \in \tau$,
\begin{enumerate}

	\item[(C1)] $\sigma \circ a(f,i)(\vec{p},a,\vec{q})  = a(f,i)(\kappa(\vec{p}),\sigma(a),\kappa(\vec{q})) \ \quad \quad \quad \quad \quad \quad \quad \left( \vec{p} \in Q^{i-1}, \vec{q} \in Q^{i}, a \in A(\alpha)/\Delta_{\alpha \alpha} \right)$
	
	$\sigma \left( f^{\Delta}(a,\vec{q}) \right) = f^{\Delta}(\sigma(a),\kappa(\vec{q})) \quad \quad \quad \quad \quad \quad \quad \quad \quad \quad \quad \quad \left( \vec{q} \in Q^{n-1}, a \in A(\alpha)/\Delta_{\alpha \alpha} \right)$
	
	\item[(C2)] $\rho \circ \sigma (x) = \kappa \circ \rho (x)    \, \quad \quad \quad  \quad \quad \quad \quad \quad \quad \quad \quad \quad \quad \quad \quad \quad \quad \quad  \left( x \in A(\alpha)/\Delta_{\alpha \alpha} \right)$
	
	\item[(C3)] $\sigma \left( \begin{bmatrix} x \\ x \end{bmatrix}/\Delta_{\alpha \alpha} \right) = \begin{bmatrix} y \\ y \end{bmatrix}/\Delta_{\alpha \alpha}$ for some $y \in A$ $ \quad \quad \quad \quad \quad \quad \quad \ \left( x \in A \right)$
	
\end{enumerate}
We see that the compatible automorphisms form a subgroup of the direct product $\Aut A(\alpha)/\Delta_{\alpha \alpha} \times \Aut Q$.

For any algebra $A$ and $\alpha \in \Con A$, define the subgroup $\Aut_{\alpha} A = \{ \phi \in \Aut A: \phi(\alpha) \subseteq \alpha \}$ of the full automorphism group of $A$. We now assume the extension $\pi: A \rightarrow Q$ realizes the affine datum $(Q,A^{\alpha,\tau},\ast)$. According to Wires \cite{wires1}, we may assume $\alpha = \ker \pi \in \Con A$. The goal is to define a homomorphism $\psi: \Aut_{\alpha} A \rightarrow C(Q,A^{\alpha,\tau},\ast)$.

For any lifting $l: Q \rightarrow A$ of $\pi$ we have 
\begin{align}\label{eqn:1}
\pi \circ l = \id_{Q} \quad \quad \quad \text{and} \quad \quad \quad \left( l \circ \pi(a), a \right) \in \alpha \quad \quad (a \in A).
\end{align}
For any $\phi \in \Aut_{\alpha} A$, define $\phi_{l}: Q \rightarrow Q$ by $\phi_{l} := \pi \circ \phi \circ l$. If $l': Q \rightarrow A$ is another lifting, then for $q \in Q$ we have $(l(q),l'(q)) \in \alpha \Rightarrow \left( \phi(l(q)),\phi(l'(q)) \right) \in \alpha$ and so $\phi_{l} = \phi_{l'}$; thus, the map $\phi_{l}$ does not depend on the choice of the lifting. If $\phi_{l}(p)=\phi_{l}(q)$, then $(\phi(l(p)),\phi(l(q))) \in \alpha = \ker \pi$ and so $(l(p),l(q)) \in \alpha$; therefore, $p = \pi(l(p))=\pi(l(q))=q$. Also, for $q \in Q$, we see that $\left( (l \circ \pi)(\phi^{-1}(l(q))), \phi^{-1}(l(q)) \right) \in \alpha$ and so  $\left( ( \phi \circ l \circ \pi)(\phi^{-1}(l(q))), l(q) \right) \in \alpha$. This implies $q = \phi_{l} \circ \pi \circ \phi^{-1} \circ l(q)$ and so $\phi_{l}$ is bijective; incidentally, we have also shown $\phi_{l}^{-1} = \pi \circ \phi^{-1} \circ l$. To show $\phi_{l}$ is a homomorphism, take $f \in \tau$, $\vec{q} \in Q^{\ar f}$ and substitute $f^{A}(l(\vec{q}))$ into $a$ in Eq~(\ref{eqn:1}) to get $\left( l(f^{Q}(\vec{q})), f^{A}(l(\vec{q})) \right) = \left( l \circ \pi( f^{A}(l(\vec{q})) ), f^{A}(l(\vec{q})) \right) \in \alpha$. Applying $\phi \in \Aut_{\alpha} A$ to the relation yields $\left( \phi \circ l(f^{Q}(\vec{q})), f^{A}( \phi \circ l (\vec{q})) \right) \in \alpha$ which then implies $\phi_{l}(f^{Q}(\vec{q})) = \pi \circ \phi \circ l(f^{Q}(\vec{q})) = \pi ( f^{A}( \phi \circ l (\vec{q})) ) = f^{Q}(\phi_{l}(\vec{q}))$; altogether, we have shown $\phi_{l} \in \Aut Q$.

For $\phi \in \Aut_{\alpha} A$, define $\widehat{\phi} : A(\alpha)/\Delta_{\alpha \alpha} \rightarrow A(\alpha)/\Delta_{\alpha \alpha}$ by $\widehat{\phi}\left( \begin{bmatrix} a \\ b \end{bmatrix}/\Delta_{\alpha \alpha} \right) := \begin{bmatrix} \phi(a) \\ \phi(b) \end{bmatrix}/\Delta_{\alpha \alpha}$. If $\begin{bmatrix} a \\ b \end{bmatrix} \mathrel \Delta_{\alpha \alpha} \begin{bmatrix} c \\ d \end{bmatrix}$, then $d = m(b,a,c)$ since $A$ realizes affine datum. Applying the automorphism $\phi$ we have $\phi(d) = m(\phi(b),\phi(a),\phi(c))$ which implies $\begin{bmatrix} \phi(a) \\ \phi(b) \end{bmatrix} \mathrel \Delta_{\alpha \alpha} \begin{bmatrix} \phi(c) \\ \phi(d) \end{bmatrix}$; thus, $\widehat{\phi}$ is well-defined. Reversing the preceding calculation and applying $\phi^{-1}$ shows injectivity, and $\begin{bmatrix} a \\ b \end{bmatrix}/\Delta_{\alpha \alpha} = \widehat{\phi} \left( \begin{bmatrix} \phi^{-1}(a) \\ \phi^{-1}(b) \end{bmatrix}/\Delta_{\alpha \alpha} \right)$ establishes surjectivity. Since $\pi: A \rightarrow Q$ realizes the datum, in the semidirect product we have by Wires \cite{wires1} for any $\alpha$-trace $r:A \rightarrow A$ and $f \in \tau$ with $n = \ar f$,
\begin{align*}
\widehat{\phi} \circ F_{f}\left( \begin{bmatrix} r(a_{1}) \\ a_{1} \end{bmatrix}/\Delta_{\alpha \alpha},\ldots,\begin{bmatrix} r(a_{n}) \\ a_{n} \end{bmatrix}/\Delta_{\alpha \alpha}  \right) &= \widehat{\phi} \left( \begin{bmatrix} f(r(a_{1}),\ldots,r(a_{n})) \\ f(a_{1},\ldots,a_{n}) \end{bmatrix}/\Delta_{\alpha \alpha} \right) \\
&= \begin{bmatrix} \phi(f(r(a),\ldots,r(a_{n}))) \\ \phi(f(a_{1},\ldots,a_{n})) \end{bmatrix}/\Delta_{\alpha \alpha} \\
&= \begin{bmatrix} f \left( \phi \circ r(a_{1}),\ldots,\phi \circ r(a_{n}) \right) \\ f(\phi(a_{1}),\ldots,\phi(a_{1})) \end{bmatrix}/\Delta_{\alpha \alpha} \\
&= F_{f} \left( \begin{bmatrix} \phi \circ r(a_{1}) \\ \phi(a_{1}) \end{bmatrix}/\Delta_{\alpha \alpha},\ldots, \begin{bmatrix} \phi \circ r(a_{1}) \\ \phi(a_{n}) \end{bmatrix}/\Delta_{\alpha \alpha}  \right) \\
&= F_{f} \left( \widehat{\phi}\left( \begin{bmatrix} r(a_{1}) \\ a_{1} \end{bmatrix}/\Delta_{\alpha \alpha} \right),\ldots,\widehat{\phi}\left( \begin{bmatrix} r(a_{n}) \\ a_{n} \end{bmatrix}/\Delta_{\alpha \alpha} \right) \right).
\end{align*}
This shows $\widehat{\phi}$ is an endomorphism of the semidirect product; altogether, $\widehat{\phi} \in \Aut A(\alpha)/\Delta_{\alpha \alpha}$.

We now define $\psi(\phi):= (\widehat{\phi},\phi_{l})$. For $\gamma, \phi \in \Aut_{\alpha} A$, clearly $\widehat{\phi \circ \gamma} = \widehat{\phi} \circ \widehat{\gamma}$. For the second-coordinate, write $a=\gamma \circ l(q)$ for $q \in Q$. Then Eq~(\ref{eqn:1}) and $\phi \in \Aut_{\alpha} A$ implies $(\phi\circ l \circ \pi(a), \phi(a)) \in \alpha$ and so $\pi \circ \phi \circ l \circ \pi(a) = \pi \circ \phi(a)$. We then see that
\[
\phi_{l} \circ \gamma_{l}(q) = \pi \circ \phi \circ l \circ \pi \circ \gamma \circ l (q) = \pi \circ \phi \circ l \circ \pi (a) = \pi \circ \phi(a) = \pi \circ \phi \circ \gamma \circ l (q) = (\phi \circ \gamma)_{l}(q);
\]
therefore, $\psi$ is a homomorphism. We can calculate how a pair of automorphisms in the image of $\psi$ interacts with the action in the semidirect product. If we fix  $f \in \tau$ with $\ar f = n > 1$ and $1 \leq i < n$, $\vec{q} =(q_{1},\ldots,q_{i}) \in Q^{i}$ and $\vec{a} = \left( \begin{bmatrix} a_{1} \\ b_{1} \end{bmatrix}/\Delta_{\alpha \alpha},\ldots, \begin{bmatrix} a_{n-i} \\ b_{n-i} \end{bmatrix}/\Delta_{\alpha \alpha} \right) \in \left( A(\alpha)/\Delta_{\alpha \alpha} \right)^{n-i}$, then by realization we have
\begin{align*}
\widehat{\phi} \left( a(f,i)(\vec{q},\vec{a}) \right) &= \widehat{\phi} \left( \begin{bmatrix} f \left( l(q_{1}),\ldots,l(q_{i}),a_{1},\ldots,a_{n-i} \right) \\ f \left( l(q_{1}),\ldots,l(q_{i}),b_{1},\ldots,b_{n-i} \right) \end{bmatrix}/\Delta_{\alpha \alpha} \right) \\
&= \begin{bmatrix} f \left( \phi(l(q_{1})),\ldots,\phi(l(q_{i})),\phi(a_{1}),\ldots,\phi(a_{n-i}) \right) \\ f \left( \phi(l(q_{1})),\ldots,\phi(l(q_{i})),\phi(b_{1}),\ldots,\phi(b_{n-i}) \right)  \end{bmatrix}/\Delta_{\alpha \alpha} = a(f,i) \left( \phi_{l}(\vec{q}), \widehat{\phi}(\vec{a}) \right).
\end{align*}
This shows $\im \psi \leq C(Q,A^{\alpha,\tau},\ast)$.

We now define the map $W_{T}: C(Q,A^{\alpha,\tau},\ast) \rightarrow  H^{2}_{\mathcal U}(Q,A^{\alpha,\tau},\ast)$ where $T$ is a 2-cocycle compatible with the variety $\mathcal U$. For each pair $(\sigma,\kappa) \in \Aut A(\alpha)/\Delta_{\alpha \alpha} \times \Aut Q$, define $T^{(\sigma,\kappa)}$ by the rule
\[
T^{(\sigma,\kappa)}_{f}(\vec{q}) := \sigma \circ T_{f}(\kappa^{-1}(\vec{q}))   \quad \quad \quad \quad \quad \quad  \left( f \in \tau, \vec{q} \in Q^{\ar f} \right).
\]
If $[T]=[S]$ as cohomology classes, then there exists $h: Q \rightarrow A(\alpha)/\Delta_{\alpha \alpha}$ such that
\[
S_{f}(\vec{q}) -_{u} T_{f}(\vec{q}) = f^{A(\alpha)/\Delta_{\alpha \alpha}}(h(\vec{q})) -_{u} h(f^{Q}(\vec{q})) \quad \quad \quad \quad \quad \left( f \in \tau, \vec{q} \in Q^{\ar f} \right) 
\]
where $u = l(f^{Q}(\vec{q}))$. Applying the automorphism $\sigma$ to the above equation and making the substitution $\vec{q} \mapsto \kappa^{-1}(\vec{q})$ we have
\[
\sigma \circ S_{f}(\kappa^{-1}(\vec{q})) -_{v} \sigma \circ T_{f}(\kappa^{-1}(\vec{q})) = f^{A(\alpha)/\Delta_{\alpha \alpha}}(\sigma \circ h \circ \kappa^{-1}(\vec{q})) -_{v} \sigma \circ h \circ \kappa^{-1} (f^{Q}(\vec{q})) 
\]
where $v=\sigma \circ l \circ \kappa^{-1}(f^{Q}(\vec{q}))$.  This shows $S^{(\sigma,\kappa)} \sim T^{(\sigma,\kappa)}$ and so $[T]^{(\sigma,\kappa)} = [T^{(\sigma,\kappa)}]$ is well-defined. If we restrict to the subgroup $C(Q,A^{\alpha,\tau},\ast)$, then the action preserves $\mathcal U$-compatibility and so defines a group action of the compatible automorphism pairs on $2^{\mathrm{nd}}$-cohomology.

To see this, take $[T] \in H^{2}_{\mathcal U}(Q,A^{\alpha,\tau},\ast)$ which means $T$ is a 2-cocycle of the datum which is compatible with the identities of $\mathcal U$. According to Wires \cite{wires1}, this means that for $t=s \in \Id \mathcal U$, the 2-cocycle $T$ satisfies for every appropriate assignment $\epsilon : \mathrm{var} \, t \cup \mathrm{var} \, s \rightarrow Q$ an equation
\begin{align}\label{eqn:2}
t^{\partial,T}(\epsilon(\mathrm{var} \, t)) = s^{\partial,T} (\epsilon(\mathrm{var} \, s))
\end{align}
For our purposes, the important thing to note is that $t^{\partial,T}(\epsilon(\mathrm{var} \, t))$ is a sum over $+_{u}$ where $u = l \left( t^{Q}(\rho(\vec{a})) \right)$. The summands are recursively defined from basic expressions of the form
\begin{align}\label{eqn:3}
a(f,i)\left(\vec{p},\omega,\vec{q} \right) \ , \quad \quad \quad g^{\Delta}(\omega,\delta \circ l(\vec{q})) \quad \quad \text{ and } \quad \quad T_{h}(\vec{p}) \quad \quad \quad \quad \quad \quad (f,g,h \in \tau)
\end{align}
by substituting each other at $\omega$ in a certain specified manner according to the composition tree of the term $t$ with the result that $\omega$ does not appear. This means that in the tree describing the resulting composition, $T_{h}(\vec{p})$ for some fundamental symbol $h \in \tau$ would appear at each leaf. Similarly for $s^{\partial,T} (\epsilon(\mathrm{var} \, s))$ and $+_{v}$ with $v = l \left( s^{Q}(\rho(\vec{b})) \right)$. In this way, we can think of Eq~(\ref{eqn:2}) as an equation satisfied by the 2-cocycle $T$.

The tuples $\vec{p}$ and $\vec{q}$ which appears in the expressions in Eq~(\ref{eqn:3}) are determined by starting with the assignment $\epsilon(\mathrm{var} \, t)$ and propagating through the composition tree of $t$ starting from the variables; consequently, $\vec{p}$ and $\vec{q}$ are the result of evaluations of different subterms calculated entirely in $Q$ starting from the assignment $\epsilon(\mathrm{var} \, t)$ of the variables. It then follows that any substitution $\epsilon(\mathrm{var} \, t) \rightarrow Q$ induces in a consistent manner simultaneous substitutions in all expressions in Eq~(\ref{eqn:3}) which were composed to form the expression $t^{\partial,T}( \epsilon(\mathrm{var} \, t))$. Given $(\sigma,\kappa) \in C(Q,A^{\alpha,\tau},\ast)$, applying $\sigma$ to the expressions in Eq~(\ref{eqn:3}), using condition (C1) and then making the substitution $\mathrm{var} \, t \mapsto \kappa^{-1}(\epsilon(\mathrm{var} \, t))$ we have the expressions 
\begin{align*}
a(f,i)\left(\vec{p}, \sigma(\omega), \vec{q} \right) \ , \quad \quad \quad g^{\Delta}(\sigma(\omega),\delta \circ l(\vec{q}))  \quad \quad \text{ and } \quad \quad \sigma \circ T_{h}(\kappa^{-1}(\vec{p})) \quad \quad \quad \quad \quad \quad (f,g,h \in \tau)
\end{align*}
which compose to give the same result as applying $\sigma$ to the expression $t^{\partial,T}(\epsilon(\mathrm{var} \, t))$ and making the substitution $\epsilon(\mathrm{var} \, t) \mapsto \kappa^{-1}(\epsilon(\mathrm{var} \, t))$. A similar discussion applies to the right-hand side of Eq~(\ref{eqn:2}). Altogether, this shows that $T^{(\sigma,\kappa)}$ satisfies the same 2-cocycle identity determined by $t=s \in \Id \mathcal U$; therefore, $[T]^{(\sigma,\kappa)} \in H^{2}_{\mathcal U}(Q,A^{\alpha,\tau},\ast)$.

We now define the map 
\[
W : C(Q,A^{\alpha,\tau},\ast) \times H^{2}_{\mathcal U}(Q,A^{\alpha,\tau},\ast) \rightarrow  H^{2}_{\mathcal U}(Q,A^{\alpha,\tau},\ast) 
\]
by $W \left( (\sigma,\kappa),[T] \right) : = [T - T^{(\sigma,\kappa)}]$. It is easy to see that it is a homomorphism in the second-coordinate. Since there is a group action of $C(Q,A^{\alpha,\tau},\ast)$ on $H^{2}_{\mathcal U}(Q,A^{\alpha,\tau},\ast)$, the restriction $W_{T} := W(-,[T])$ is a principal derivation of the corresponding group datum, and is called the \emph{Wells derivation}. As a derivation, we have $\ker W_{T} \leq C(Q,A^{\alpha,\tau},\ast)$ but not generally as a normal subgroup. Observe that $\ker W_{T} \cap \ker W_{T'} \leq \ker W_{T + T'}$; therefore, it follows that $\ker W = \{[T] : \ker W_{T} = C(Q,A^{\alpha,\tau},\ast) \}$ is a subgroup of $H^{2}_{\mathcal U} \left( Q, A^{\alpha,\tau},\ast \right)$.

\vspace{0.5cm}


\section{Demonstrations for Theorem~\ref{thm:2} and Theorem~\ref{thm:3}}
\vspace{0.3cm}

In this section, we give the proofs of Theorem~\ref{thm:2} and Theorem~\ref{thm:3}.

\begin{proof} (of Theorem \ref{thm:2})
Fix an extension $\pi:A \rightarrow Q$ realizing affine datum $(Q,A^{\alpha,\tau},\ast)$ with $A \in \mathcal U$ and let $[T] \in H^{2}_{\mathcal U}(Q,A^{\alpha,\tau},\ast)$ by the 2-cocycle compatible with $\mathcal U$ associated to the extension. Fix a lifting $l: Q \rightarrow A$ for $\pi$. We see that $\phi \in \ker \psi$ if and only if the two conditions 
\begin{align}\label{eqn:5}
\widehat{\phi}=\id_{A(\alpha)/\Delta_{\alpha \alpha}} \quad \quad \quad \text{and} \quad \quad \quad \pi \circ \phi \circ l = \phi_{l} = \id_{Q}
\end{align}
on $\phi$ hold. Assume $\phi \in \ker \psi$. For any $a \in A$, $(l \circ \pi(a),a) \in \alpha$ implies $(\phi \circ l \circ \pi (a),\phi(a)) \in \alpha$ which yields $\pi \circ \phi (a) = \pi \circ \phi \circ l \circ \pi (a) = \pi (a)$ by the second condition on $\phi$; that is,
\begin{align}\label{eqn:6}
\pi \circ \phi = \pi.
\end{align}
Now according to Wires \cite{wires1}, every element in the universe $A(\alpha)/\Delta_{\alpha \alpha}$ is uniquely represented in the form $\begin{bmatrix} r(a) \\ a \end{bmatrix}/\Delta_{\alpha \alpha}$ for any $\alpha$-trace $r : A \rightarrow A$. Since $(Q,A^{\alpha,\tau},\ast)$ is affine datum, the first condition in Eq~(\ref{eqn:5}) yields $\begin{bmatrix} \phi(r(a)) \\ \phi(a) \end{bmatrix}/\Delta_{\alpha \alpha} = \widehat{\phi} \left( \begin{bmatrix} r(a) \\ a \end{bmatrix}/\Delta_{\alpha \alpha} \right) = \begin{bmatrix} r(a) \\ a \end{bmatrix}/\Delta_{\alpha \alpha}$ which implies 
\begin{align}\label{eqn:7}
\phi(a) = m(\phi(r(a)),r(a),a) \quad \quad \quad  \text{for any} \ \alpha-\text{trace} \ r: A \rightarrow A. 
\end{align}
According to Wires \cite{wires1}, Eq~(\ref{eqn:6}) and Eq~(\ref{eqn:7}) are precisely the conditions for $\phi \in \mathrm{Stab}(\pi: A \rightarrow Q)$. Working the above argument in reverse shows Eq~(\ref{eqn:6}) and Eq~(\ref{eqn:7}) implies Eq~(\ref{eqn:5}); altogether, $\mathrm{Stab}(\pi: A \rightarrow Q) = \ker \psi$. Then by Wires \cite{wires1}, this yields an embedding $\mathrm{Der}(Q,A^{\alpha,\tau},\ast) \approx \mathrm{Stab}(\pi: A \rightarrow Q) = \ker \psi \leq \Aut_{\alpha} A$.

We now verify exactness at $C(Q,A^{\alpha,\tau},\ast)$. Given $\phi \in \Aut_{\alpha} A$ and $q \in Q$, observe $\pi \circ \phi \circ l \circ \phi_{l}^{-1}(q) = \phi_{l} \circ \phi_{l}^{-1}(q) = q$. This implies $\phi \circ l \circ \phi_{l}^{-1} : Q \rightarrow A$ is another lifting for $\pi$. Since $\pi:A \rightarrow Q$ realizes the datum, we see that for $f \in \tau$ and $\vec{q} \in Q^{\ar f}$,
\begin{align*}
 T^{(\widehat{\phi},\phi_{l})}(\vec{q}) = \widehat{\phi} \circ T_{f}(\phi_{l}^{-1}(\vec{q})) = \widehat{\phi} \left( \begin{bmatrix}  l \left( f^{Q}(\phi_{l}^{-1}(\vec{q})) \right) \\ f^{A} \left( l \circ \phi_{l}^{-1}(\vec{q}) \right)  \end{bmatrix}/\Delta_{\alpha \alpha} \right) = \begin{bmatrix}  \phi \circ l \circ \phi_{l}^{-1}\left( f^{Q}(\vec{q}) \right) \\ f^{A} \left( \phi \circ l \circ \phi_{l}^{-1}(\vec{q}) \right) \end{bmatrix}/\Delta_{\alpha \alpha}.
\end{align*}
This implies $T^{(\widehat{\phi},\phi_{l})}$ is the 2-cocycle defined by the lifting $\phi \circ l \circ \phi_{l}^{-1}$. According to Wires \cite{wires1}, we see that $T \sim T^{(\widehat{\phi},\phi_{l})}$ and so $W_{T} \circ \psi(\phi) = [T - T^{(\widehat{\phi},\phi_{l})}] = 0$.

Now take $(\sigma,\kappa) \in C(Q,A^{\alpha,\tau},\ast)$ and suppose $W_{T}(\sigma,\kappa) = 0$ which implies $T \sim T^{(\sigma,\kappa)}$. So there exists $h:Q \rightarrow A(\alpha)/\Delta_{\alpha \alpha}$ such that
\begin{align}\label{eqn:12}
\sigma \circ T_{f}(\kappa^{-1}(\vec{q})) -_{w} T_{f}(\vec{q}) = f^{A(\alpha)/\Delta_{\alpha \alpha}}(h(\vec{q})) -_{w} h(f^{Q}(\vec{q})) \quad \quad \quad \quad \quad \quad \left( f \in \tau, \vec{q} \in Q^{\ar f} \right)
\end{align}
where $w = l(f^{Q}(\vec{q}))$. We use this to define $\phi \in \Aut_{\alpha} A$ in the following manner. Fix the $\alpha$-trace $r = l \circ \pi$. By Wires \cite{wires1}, there is an isomorphism
$\gamma : A \rightarrow A_{T}(Q,A^{\alpha,\tau},\ast)$ given by $\gamma(x) = \begin{bmatrix} r(x) \\ x \end{bmatrix}/\Delta_{\alpha \alpha}$ where $A_{T}(Q,A^{\alpha,\tau},\ast)$ is the extension directly reconstructed from the datum. It will be useful to record  
\begin{align}
\rho \circ \gamma = \pi     \quad \quad \quad  \quad \quad \quad    \rho \circ h = \id_{Q}
\end{align}
for the following calculations. Define $\bar{\phi} : A_{T}(Q,A^{\alpha,\tau},\ast) \rightarrow A_{T}(Q,A^{\alpha,\tau},\ast)$ by the rule
\begin{align}\label{eqn:10}
\bar{\phi} \left( \begin{bmatrix} r(x) \\ x \end{bmatrix}/\Delta_{\alpha \alpha} \right) := \sigma \left( \begin{bmatrix} r(x) \\ x \end{bmatrix}/\Delta_{\alpha \alpha} \right) +_{u} h \left( \kappa \circ \rho \left( \begin{bmatrix} r(x) \\ x \end{bmatrix}/\Delta_{\alpha \alpha} \right) \right)
\end{align}
where $u = l \circ \rho \circ \sigma \left( \begin{bmatrix} r(x) \\ x \end{bmatrix}/\Delta_{\alpha \alpha} \right)$. Then $\phi := \gamma^{-1} \circ \bar{\phi} \circ \gamma \in \Aut A_{\alpha}$ will be the automorphism we seek such that $\psi(\phi) = (\sigma,\kappa)$.

We first show $\phi_{l} = \kappa$. Define $\sigma' : A \rightarrow A$ by $\sigma' = \gamma^{-1} \circ \sigma \circ \gamma$. Then $u = l \circ \rho \circ \sigma \circ \gamma = l \circ \rho \circ \gamma \circ \gamma^{-1} \circ \sigma \circ \gamma = r \circ \gamma^{-1} \circ \sigma \circ \gamma = r \circ \sigma'$. This implies we can write $\sigma \left( \begin{bmatrix} r(x) \\ x \end{bmatrix}/\Delta_{\alpha \alpha} \right) = \sigma \circ \gamma(x) = \begin{bmatrix} u \\ \sigma'(x) \end{bmatrix}/\Delta_{\alpha \alpha}$. We note that $r(u)=u$ and thus, $\gamma (u) = \begin{bmatrix} r(u) \\ u\end{bmatrix}/\Delta_{\alpha \alpha} = \begin{bmatrix} u \\ u\end{bmatrix}/\Delta_{\alpha \alpha}$. We can also see that $\sigma \left( \begin{bmatrix} r(x) \\ x \end{bmatrix}/\Delta_{\alpha \alpha} \right) = \begin{bmatrix} u \\ \sigma'(x) \end{bmatrix}/\Delta_{\alpha \alpha} \ \mathrel{\hat{\alpha}/\Delta_{\alpha \alpha}} \ \begin{bmatrix} u \\ u \end{bmatrix}/\Delta_{\alpha \alpha}$ by definition of $u \in A$; thus, $\rho \circ \sigma \circ \gamma(x) = \rho \circ \gamma (u)$. Recall that since $(\sigma,\kappa)$ are compatible, $\rho \circ \sigma \circ \gamma (x) = \kappa \circ \rho \circ \gamma(x)$. Then by using the idempotence of the term $m$ in Q, we have
\begin{align*}
\pi \circ \phi (x) = \pi \circ \gamma^{-1} \circ \bar{\phi} \circ \gamma (x) &= \pi \circ \gamma^{-1} \left( \sigma \circ \gamma(x) +_{u} h \left( \kappa \circ \rho \circ \gamma(x) \right)  \right) \\ 
&= \rho \circ m^{A(\alpha)/\Delta_{\alpha \alpha}} \left( \sigma \circ \gamma(x), \gamma(u), h \left( \kappa \circ \rho \circ \gamma(x) \right) \right) \\
&= m^{Q}\left( \rho \circ \sigma \circ \gamma(x), \rho \circ \gamma(u), \kappa \circ \rho \circ \gamma(x)  \right) \\
&= m^{Q}\left( \rho \circ \sigma \circ \gamma(x), \rho \circ \gamma(u), \rho \circ \sigma \circ \gamma(x)  \right) \\
&= \rho \circ \sigma \circ \gamma(x) = \kappa \circ \rho \circ \gamma (x) = \kappa \circ \pi (x);
\end{align*}
that is,
\begin{align}\label{eqn:9}
\pi \circ \phi = \kappa \circ \pi.
\end{align}
Then $\phi_{l} = \pi \circ \phi \circ l = \kappa \circ \pi \circ l = \kappa \circ \id_{Q} = \kappa$ as desired.

Let us now show $\widehat{\phi} = \sigma$. If we define $h': Q \rightarrow A$ by $h' = \gamma^{-1} \circ h$, then we can write $h(\pi(x)) = \begin{bmatrix} r \circ h(\pi(x)) \\ h'(\pi(x)) \end{bmatrix}/\Delta_{\alpha \alpha}$. Since $l \circ \rho \circ h(\kappa \circ \pi(x) ) = l \circ \rho \circ h(\kappa \circ \rho \circ \gamma(x) ) = l \circ \kappa \circ \rho \circ \gamma(x) = l \circ \rho \circ \sigma \circ \gamma (x)=u$, we see that $h(\kappa \circ \pi(x)) = \begin{bmatrix} u \\ h'(\kappa \circ \pi(x)) \end{bmatrix}/\Delta_{\alpha \alpha}$. At this point, we should note $h'(\kappa \circ \pi(x)) = h'(\kappa \circ \pi(r(x))$. We can also give a representation of $\bar{\phi}$ by
\begin{align}\label{eqn:8}
\bar{\phi} \left( \begin{bmatrix} r(x) \\ x \end{bmatrix}/\delta_{\alpha \alpha} \right) = \begin{bmatrix} u \\ \sigma'(x)  \end{bmatrix}/\Delta_{\alpha \alpha} +_{u} \begin{bmatrix} u \\ h'(\kappa \circ \pi(x))  \end{bmatrix}/\Delta_{\alpha \alpha} = \begin{bmatrix} u \\ m^{A} \left( \sigma'(x), u, h'(\kappa \circ \pi(x)) \right)  \end{bmatrix}/\Delta_{\alpha \alpha}
\end{align}
which yields $\phi(x) = m^{A} \left( \sigma'(x), u, h'(\kappa \circ \pi(x)) \right)$. Since $(\sigma,\kappa)$ are compatible, we have $\rho \circ \sigma \circ \gamma(x) = \kappa \circ \rho \circ \gamma (x)=  \kappa \circ \rho \circ \gamma(r(x))  = \rho \circ \sigma \circ \gamma(r(x))$. This implies $u = l \circ \rho \circ \sigma \circ \gamma(x) = l \circ \rho \circ \sigma \circ \gamma(r(x)) = l \circ \pi \circ \gamma^{-1} \circ \sigma \circ \gamma(r(x)) = r \circ \sigma'(r(x))$; therefore, since $(\sigma,\kappa)$ are compatible, it must be that $\sigma \left( \begin{bmatrix} r(x) \\ r(x) \end{bmatrix}/\Delta_{\alpha \alpha} \right) = \begin{bmatrix} r \circ \sigma'(r(x)) \\ \sigma'(r(x)) \end{bmatrix}/\Delta_{\alpha \alpha} = \begin{bmatrix} u \\ \sigma'(r(x)) \end{bmatrix}/\Delta_{\alpha \alpha}$ is a $\Delta_{\alpha \alpha}$-class of a diagonal. Since this element is unique in any $\hat{\alpha}/\Delta_{\alpha \alpha}$-class because $\alpha$ is abelian, it must be that $\sigma'(r(x)) = u$. We can now calculate
\begin{align*}
\widehat{\phi} \left( \begin{bmatrix} r(x) \\ x \end{bmatrix}/\Delta_{\alpha \alpha} \right) = \begin{bmatrix} \phi(r(x)) \\ \phi(x) \end{bmatrix}/\Delta_{\alpha \alpha}  &= \begin{bmatrix} m^{A} \left( \sigma'(r(x)), u, h'(\kappa \circ \pi(r(x))) \right) \\ m^{A} \left( \sigma'(x), u, h'(\kappa \circ \pi(x)) \right) \end{bmatrix}/\Delta_{\alpha \alpha} \\
&= m^{A(\alpha)/\Delta_{\alpha \alpha}} \left( \begin{bmatrix} \sigma'(r(x)) \\ \sigma'(x) \end{bmatrix}/\Delta_{\alpha \alpha}, \begin{bmatrix} u \\ u \end{bmatrix}/\Delta_{\alpha \alpha}, \begin{bmatrix} u \\ u \end{bmatrix}/\Delta_{\alpha \alpha}  \right) \\
&= \begin{bmatrix} \sigma'(r(x)) \\ \sigma'(x) \end{bmatrix}/\Delta_{\alpha \alpha} = \begin{bmatrix} u \\ \sigma'(x) \end{bmatrix}/\Delta_{\alpha \alpha} = \sigma \left( \begin{bmatrix} r(x) \\ x \end{bmatrix}/\Delta_{\alpha \alpha} \right)
\end{align*}
since $m$ is Mal'cev on $\hat{\alpha}/\Delta_{\alpha \alpha}$-classes and so $\widehat{\phi} = \sigma$; altogether, we have shown $\psi(\phi) = (\sigma,\kappa)$.

Let us now show $\phi$ preserves $\alpha$ and is bijective. If $(a,b) \in \alpha$, then by Eq~(\ref{eqn:9}) we have $\pi \circ \phi(a) = \kappa \circ \pi(a) = \kappa \circ \pi(b) = \pi \circ \phi(b)$ and so $(\phi(a),\phi(b)) \in \alpha$; thus, $\phi(\alpha) \subseteq \alpha$.

We show $\phi$ is injective. If $\phi(a) = \phi(b)$, then $\kappa \circ \pi(a) = \pi \circ \phi(a) = \pi \circ \phi(b) = \kappa \circ \pi(b)$ which implies $\pi(a) = \pi(b)$ since $\kappa \in \Aut Q$; therefore, $r(a)=r(b)$. We also see that $h(\kappa \circ \rho \circ \gamma(a)) = h( \kappa \circ \pi(a)) = h( \kappa \circ \pi(b)) = h(\kappa \circ \rho \circ \gamma(b))$ and $l \circ \rho \circ \sigma \circ \gamma(a) = l \circ \kappa \circ \pi(a) = l \circ \kappa \circ \pi(b) = l \circ \rho \circ \sigma \circ \gamma(b)$. Using these facts in Eq~(\ref{eqn:10}) yields $\sigma \left( \begin{bmatrix} r(a) \\ a \end{bmatrix}/\Delta_{\alpha \alpha} \right) = \sigma \left( \begin{bmatrix} r(b) \\ b \end{bmatrix}/\Delta_{\alpha \alpha} \right)$ and so $\begin{bmatrix} r(a) \\ a \end{bmatrix}/\Delta_{\alpha \alpha} = \begin{bmatrix} r(b) \\ b \end{bmatrix}/\Delta_{\alpha \alpha}$ since $\sigma$ is an automorphism. Then $r(a)=r(b)$ implies $a=b$.

To show $\phi$ is surjective, it suffices to show $\bar{\phi}$ is. Given $a \in A$, set $b=\sigma^{-1} \left( \begin{bmatrix} r(a) \\ a \end{bmatrix}/\Delta_{\alpha \alpha} -_{r(a)} h(\pi(a)) \right)$. Observe that 
\begin{align}\label{eqn:11}
h \left( k \circ \rho \circ \sigma^{-1} \left( \begin{bmatrix} r(a) \\ a \end{bmatrix}/\Delta_{\alpha \alpha} -_{r(a)} h(\pi(a)) \right) \right) = h \left( \rho \left( \begin{bmatrix} r(a) \\ a \end{bmatrix}/\Delta_{\alpha \alpha} -_{r(a)} h(\pi(a)) \right) \right) = h(\pi(a))
\end{align}
by idempotence of $m^{Q}$; in particular, $l \circ \rho \circ \sigma (b) = r(a)$. Then using Eq~(\ref{eqn:11}) we have 
\begin{align*}
\bar{\phi}(b) = \sigma(b) +_{r(a)} h(\pi(a)) = \begin{bmatrix} r(a) \\ a \end{bmatrix}/\Delta_{\alpha \alpha} -_{r(a)} h(\pi(a)) +_{r(a)} h(\pi(a)) = \begin{bmatrix} r(a) \\ a \end{bmatrix}/\Delta_{\alpha \alpha}; 
\end{align*}
therefore; $\hat{\phi}$ is surjective.

The last task is to show $\bar{\phi}$ is a homomorphism which will imply $\phi$ is, as well. In Eq~(\ref{eqn:12}), we make the substitution $\vec{q} \mapsto \kappa(\vec{q})$ and rewrite to conclude
\begin{align}\label{eqn:13}
\sigma \circ T_{f}(\vec{q}) +_{v} h \left( f^{Q}(\kappa(\vec{q})) \right) = f^{A(\alpha)/\Delta_{\alpha \alpha}} \left( h(\kappa(\vec{q})) \right) + _{v} T_{f}(\kappa(\vec{q}))
\end{align}
where $v=l \left( f^{Q}(\kappa(\vec{q})) \right)$. Let us observe that $l \circ \rho \circ \sigma \left( \begin{bmatrix} r(f(\vec{x})) \\ f(\vec{x}) \end{bmatrix}/\Delta_{\alpha \alpha} \right) = l \circ \kappa \circ \rho \left( \begin{bmatrix} r(f(\vec{x})) \\ f(\vec{x}) \end{bmatrix}/\Delta_{\alpha \alpha} \right) = l \left( f^{Q}(\kappa(\vec{q})) \right) = v$ where we have written $\pi(\vec{x}) = \vec{q}$. Then by realization of the datum, we have
\begin{align*}
\bar{\phi} \left( F_{f} \left( \begin{bmatrix} r(\vec{x}) \\ \vec{x} \end{bmatrix}/\Delta_{\alpha \alpha} \right) \right) &= \bar{\phi} \left( \begin{bmatrix} r(f(\vec{x})) \\ f(\vec{x}) \end{bmatrix}/\Delta_{\alpha \alpha} \right) \\
&= \sigma \left( \begin{bmatrix} r(f(\vec{x})) \\ f(\vec{x}) \end{bmatrix}/\Delta_{\alpha \alpha} \right) +_{v} h \left( \kappa \circ \rho \left( \begin{bmatrix} r(f(\vec{x})) \\ f(\vec{x}) \end{bmatrix}/\Delta_{\alpha \alpha} \right) \right) \\
&= \sigma \circ f^{A(\alpha)/\Delta_{\alpha \alpha}} \left( \begin{bmatrix} r(\vec{x}) \\ \vec{x} \end{bmatrix}/\Delta_{\alpha \alpha} \right)  +_{v} \sigma \circ T_{f}(\vec{q}) +_{v} h \left( \kappa(f^{Q}(\vec{q})) \right) \\
&= f^{A(\alpha)/\Delta_{\alpha \alpha}} \left( \sigma \left( \begin{bmatrix} r(\vec{x}) \\ \vec{x} \end{bmatrix}/\Delta_{\alpha \alpha} \right) \right)  +_{v} f^{A(\alpha)/\Delta_{\alpha \alpha}} \left( h(\kappa(\vec{q})) \right) +_{v} T_{f}\left( \kappa(\vec{q}) \right)  \\
&= f^{A(\alpha)/\Delta_{\alpha \alpha}} \left( \bar{\phi}(\vec{x}) \right) +_{v} T_{f}(\kappa(\vec{q})) \\
&= F_{f} \left( \bar{\phi}(\vec{x}) \right).
\end{align*}
The demonstration is now complete.
\end{proof}

\begin{proof} (of Theorem \ref{thm:3})
Given $[T] \in H^{2}_{\mathcal U}(Q,A^{\alpha,\tau},\ast)$, there is a realization $\pi: A_{T}(Q,A^{\alpha,\tau},\ast) \rightarrow Q$ with $A_{T}(Q,A^{\alpha,\tau},\ast) \in \mathcal U$. For each $(\sigma,\kappa) \in \ker W_{T}$, there is a function $h_{(\sigma,\kappa)} : Q \rightarrow A(\alpha)/\Delta_{\alpha \alpha}$ such that
\begin{align*}
T^{(\sigma,\kappa)}_{f}(\vec{q}) -_{v} T_{f}(\vec{q}) = f^{A(\alpha)/\Delta_{\alpha \alpha}} \left( h_{(\sigma,\kappa)}(\vec{q}) \right) -_{v} h_{(\sigma,\kappa)} \left( f^{Q}(\vec{q}) \right) \quad \quad \quad \quad \quad \left( f \in \tau, \vec{q} \in Q^{\ar f} \right)
\end{align*}
where $v = l(f^{Q}(\vec{q}))$ for a lifting $l$ of $\pi$. We follow Eq~(\ref{eqn:10}) and define
\begin{align}
\hat{l}_{T}(\sigma,\kappa) \left( \begin{bmatrix} r(x) \\ x \end{bmatrix}/\Delta_{\alpha \alpha} \right) := \sigma \left( \begin{bmatrix} r(x) \\ x \end{bmatrix}/\Delta_{\alpha \alpha}\right) +_{u} h_{(\sigma,\kappa)} \left( \kappa \circ \pi \left( \begin{bmatrix} r(x) \\ x \end{bmatrix}/\Delta_{\alpha \alpha}  \right)  \right)
\end{align}
for $(\sigma,\kappa) \in \ker W_{T}$ where $u=l \circ \pi \circ \sigma \left( \begin{bmatrix} r(x) \\ x \end{bmatrix}/\Delta_{\alpha \alpha}  \right)$ and $\alpha$-trace $r = l \circ \pi$. As in the last part of the proof of Theorem~\ref{thm:2}, we can see that $\hat{l}_{T} : \ker W_{T} \rightarrow \Aut_{\hat{\alpha}} A_{T}(Q,A^{\alpha,\tau},\ast)$ and is a lifting for $\psi : A_{T}(Q,A^{\alpha,\tau},\ast) \rightarrow C(Q,A^{\alpha,\tau},\ast)$. For $d \in \mathrm{Der}(Q,A^{\alpha,\tau},\ast)$, we let $\overline{d}$ denote the corresponding stabilizing automorphism so that 
\begin{align}
\overline{d} \left( \begin{bmatrix} r(x) \\ x \end{bmatrix}/\Delta_{\alpha \alpha} \right) = \begin{bmatrix} r(x) \\ x \end{bmatrix}/\Delta_{\alpha \alpha} +_{r(x)} d \circ \pi \left( \begin{bmatrix} r(x) \\ x \end{bmatrix}/\Delta_{\alpha \alpha} \right)  .
\end{align}
Now define the 2-cocycle $S$ determined by the lifting $\hat{l}$ for the extension $\psi$ so that $\overline{S \left( ( x,y), (u,v) \right)}:= \hat{l}_{T}(x,y) \circ \hat{l}_{T}(u,v) \circ \hat{l}_{T}(xu,yv)^{-1}$. Theorem~\ref{thm:2} and group cohomology yields the isomorphism 
\begin{align*}
\Aut_{\hat{\alpha}} A(\alpha,\ast,T) \approx \mathrm{Der}(Q,A^{\alpha,\tau},\ast) \rtimes_{(S,\phi)} \ker W_{T}
\end{align*}
where the action $\phi : \ker W_{T} \rightarrow \Aut \mathrm{Der}(Q,A^{\alpha,\tau},\ast)$ is induced by the lifting $\hat{l}$ of $\psi$ in the standard manner. The task now is to show $[S]=0$ in cohomology which is done by first evaluating the 2-cocycle $S$.

Let us first note that
\[
\hat{l}_{T}(\gamma \sigma,\beta \kappa)^{-1} \left( \begin{bmatrix} r(x) \\ x \end{bmatrix}/\Delta_{\alpha \alpha} \right) = \sigma^{-1} \circ \gamma^{-1} \left( \left( \begin{bmatrix} r(x) \\ x \end{bmatrix}/\Delta_{\alpha \alpha} \right) -_{r(x)} h_{(\gamma \sigma,\beta \kappa)} \left( \pi \left( \begin{bmatrix} r(x) \\ x \end{bmatrix}/\Delta_{\alpha \alpha} \right)  \right) \right).
\]
Using compatibility of $(\gamma,\beta)$, we can simplify
\begin{align*}
u &= l \circ \pi \circ \sigma \left( \hat{l}_{T}(\gamma \sigma,\beta \kappa)^{-1} \left( \begin{bmatrix} r(x) \\ x \end{bmatrix}/\Delta_{\alpha \alpha} \right)  \right) \\
&= l \circ \pi \circ \gamma^{-1} \left( \left( \begin{bmatrix} r(x) \\ x \end{bmatrix}/\Delta_{\alpha \alpha} \right) -_{r(x)} h_{(\gamma \sigma,\beta \kappa)} \left( \pi \left( \begin{bmatrix} r(x) \\ x \end{bmatrix}/\Delta_{\alpha \alpha} \right)  \right) \right) \\
&= l \circ \beta^{-1} \circ \pi \left( \left( \begin{bmatrix} r(x) \\ x \end{bmatrix}/\Delta_{\alpha \alpha} \right) -_{r(x)} h_{(\gamma \sigma,\beta \kappa)} \left( \pi \left( \begin{bmatrix} r(x) \\ x \end{bmatrix}/\Delta_{\alpha \alpha} \right)  \right) \right) \\
&= l \circ \beta^{-1} \circ \pi \left( \begin{bmatrix} r(x) \\ x \end{bmatrix}/\Delta_{\alpha \alpha} \right)
\end{align*}
and
\begin{align*}
h_{(\sigma,\kappa)} \left( \kappa \circ \pi \left(  \hat{l}_{T}(\gamma \sigma,\beta \kappa)^{-1} \left( \begin{bmatrix} r(x) \\ x \end{bmatrix}/\Delta_{\alpha \alpha} \right) \right) \right) &= h_{(\sigma,\kappa)} \left( \pi \circ \sigma \left(  \hat{l}_{T}(\gamma \sigma,\beta \kappa)^{-1} \left( \begin{bmatrix} r(x) \\ x \end{bmatrix}/\Delta_{\alpha \alpha} \right) \right) \right) \\
&= h_{(\sigma,\kappa)} \left( \beta^{-1} \circ \pi \left( \begin{bmatrix} r(x) \\ x \end{bmatrix}/\Delta_{\alpha \alpha} \right) \right).
\end{align*}
Then
\begin{align*}
\hat{l}_{T}(\sigma,\kappa) \circ \hat{l}_{T}(\gamma \sigma,\beta \kappa)^{-1} \left( \begin{bmatrix} r(x) \\ x \end{bmatrix}/\Delta_{\alpha \alpha} \right) &= \gamma^{-1} \left( \left( \begin{bmatrix} r(x) \\ x \end{bmatrix}/\Delta_{\alpha \alpha} \right) -_{r(x)} h_{(\gamma \sigma,\beta \kappa)} \left( \pi \left( \begin{bmatrix} r(x) \\ x \end{bmatrix}/\Delta_{\alpha \alpha} \right)  \right) \right) \\ 
&\quad +_{u} h_{(\sigma,\kappa)} \left( \beta^{-1} \circ \pi \left( \begin{bmatrix} r(x) \\ x \end{bmatrix}/\Delta_{\alpha \alpha} \right) \right).
\end{align*}
Again, we can calculate in a similar manner
\begin{align*}
w &= l \circ \pi \circ \gamma \left( \hat{l}_{T}(\sigma,\kappa) \circ \hat{l}_{T}(\gamma \sigma,\beta \kappa)^{-1} \left( \begin{bmatrix} r(x) \\ x \end{bmatrix}/\Delta_{\alpha \alpha} \right) \right) \displaybreak[0]\\
&= l \circ \beta \circ \pi \left( \hat{l}_{T}(\sigma,\kappa) \circ \hat{l}_{T}(\gamma \sigma,\beta \kappa)^{-1} \left( \begin{bmatrix} r(x) \\ x \end{bmatrix}/\Delta_{\alpha \alpha} \right) \right)  \displaybreak[0]\\
&= l \circ \beta \left( \beta^{-1} \circ \pi \left( \begin{bmatrix} r(x) \\ x \end{bmatrix}/\Delta_{\alpha \alpha} \right) \right) = r(x) \\
\end{align*}
and
\begin{align*}
z = l \circ \pi \circ \gamma \left( \begin{bmatrix} u \\ u \end{bmatrix}/\Delta_{\alpha \alpha} \right) = l \circ \beta \circ \pi \left(  \begin{bmatrix} u \\ u \end{bmatrix}/\Delta_{\alpha \alpha} \right) = l \circ \beta^{-1} \circ \beta \circ \pi \left( \begin{bmatrix} r(x) \\ x \end{bmatrix}/\Delta_{\alpha \alpha} \right) = r(x).
\end{align*}
Altogether,  
\begin{align*}
\overline{S \left( (\gamma,\beta), (\sigma,\kappa) \right)} \left( \begin{bmatrix} r(x) \\ x \end{bmatrix}/\Delta_{\alpha \alpha} \right) &= \gamma \left( \hat{l}_{T}(\sigma,\kappa) \circ \hat{l}_{T}(\gamma \sigma,\beta \kappa)^{-1} \left( \begin{bmatrix} r(x) \\ x \end{bmatrix}/\Delta_{\alpha \alpha} \right)  \right) +_{w} \displaybreak[0]\\
&\, h_{(\gamma,\beta)}\left( \beta \circ \pi \left( \hat{l}_{T}(\sigma,\kappa) \circ \hat{l}_{T}(\gamma \sigma,\beta \kappa)^{-1} \left( \begin{bmatrix} r(x) \\ x \end{bmatrix}/\Delta_{\alpha \alpha} \right)  \right)  \right) \displaybreak[0]\\
&= \begin{bmatrix} r(x) \\ x \end{bmatrix}/\Delta_{\alpha \alpha} -_{r(x)}  h_{(\gamma \sigma,\beta \kappa)} \left( \pi \left( \begin{bmatrix} r(x) \\ x \end{bmatrix}/\Delta_{\alpha \alpha} \right)  \right) +_{r(x)} \displaybreak[0]\\
&\, \gamma \circ  h_{(\sigma,\kappa)} \left( \beta^{-1} \circ \pi \left( \begin{bmatrix} r(x) \\ x \end{bmatrix}/\Delta_{\alpha \alpha} \right) \right)  +_{r(x)} \displaybreak[0]\\
&\, h_{(\gamma,\beta)} \left( \pi \left( \begin{bmatrix} r(x) \\ x \end{bmatrix}/\Delta_{\alpha \alpha} \right) \right).
\end{align*}
The next step is to evaluate the action $\phi(\sigma,\kappa): \mathrm{Der}(Q,A^{\alpha,\tau},\ast) \rightarrow \mathrm{Der}(Q,A^{\alpha,\tau},\ast)$. If we write $w = l \circ \kappa^{-1} \circ \pi \left( \begin{bmatrix} r(x) \\ x \end{bmatrix}/\Delta_{\alpha \alpha} \right)$, then the action can be evaluated by 
\begin{align*}
\overline{\phi(\sigma,\kappa)(d)} &\left( \begin{bmatrix} r(x) \\ x \end{bmatrix}/\Delta_{\alpha \alpha} \right)  \\
&= \hat{l}_{T}(\sigma,\kappa) \circ \overline{d} \circ \hat{l}_{T}(\sigma,\kappa)^{-1} \left( \begin{bmatrix} r(x) \\ x \end{bmatrix}/\Delta_{\alpha \alpha} \right) \\
&= \hat{l}_{T}(\sigma,\kappa) \Bigg( \sigma^{-1} \left( \begin{bmatrix} r(x) \\ x \end{bmatrix}/\Delta_{\alpha \alpha} -_{r(x)} h_{(\sigma,\kappa)} \circ \pi \left( \begin{bmatrix} r(x) \\ x \end{bmatrix}/\Delta_{\alpha \alpha} \right)  \right) \\
&\quad +_{w} d \circ \pi \circ \sigma^{-1} \left( \begin{bmatrix} r(x) \\ x \end{bmatrix}/\Delta_{\alpha \alpha} -_{r(x)} h_{(\sigma,\kappa)} \circ \pi \left( \begin{bmatrix} r(x) \\ x \end{bmatrix}/\Delta_{\alpha \alpha} \right)  \right)  \Bigg)   \\
&= \hat{l}_{T}(\sigma,\kappa) \Bigg( \sigma^{-1} \left( \begin{bmatrix} r(x) \\ x \end{bmatrix}/\Delta_{\alpha \alpha} -_{r(x)} h_{(\sigma,\kappa)} \circ \pi \left( \begin{bmatrix} r(x) \\ x \end{bmatrix}/\Delta_{\alpha \alpha} \right)  \right) \\
&\quad +_{w} d \circ \kappa^{-1} \circ \pi \left( \begin{bmatrix} r(x) \\ x \end{bmatrix}/\Delta_{\alpha \alpha} \right)  \Bigg)   \\
&=  \begin{bmatrix} r(x) \\ x \end{bmatrix}/\Delta_{\alpha \alpha} -_{r(x)} h_{(\sigma,\kappa)} \circ \pi \left( \begin{bmatrix} r(x) \\ x \end{bmatrix}/\Delta_{\alpha \alpha} \right) +_{r(x)} \sigma \circ d \left( \kappa^{-1} \circ \pi \left( \begin{bmatrix} r(x) \\ x \end{bmatrix}/\Delta_{\alpha \alpha} \right)   \right)  \\
&\quad +_{r(x)} h_{(\sigma,\kappa)} \left( \kappa \circ \pi \circ \overline{d} \circ \hat{l}_{T}(\sigma,\kappa)^{-1} \left( \begin{bmatrix} r(x) \\ x \end{bmatrix}/\Delta_{\alpha \alpha} \right)   \right) \\
&= \begin{bmatrix} r(x) \\ x \end{bmatrix}/\Delta_{\alpha \alpha} -_{r(x)} h_{(\sigma,\kappa)} \circ \pi \left( \begin{bmatrix} r(x) \\ x \end{bmatrix}/\Delta_{\alpha \alpha} \right) +_{r(x)} \sigma \circ d \left( \kappa^{-1} \circ \pi \left( \begin{bmatrix} r(x) \\ x \end{bmatrix}/\Delta_{\alpha \alpha} \right)   \right)  \\
&\quad +_{r(x)} h_{(\sigma,\kappa)} \circ \pi \left( \begin{bmatrix} r(x) \\ x \end{bmatrix}/\Delta_{\alpha \alpha} \right) \\
&= \begin{bmatrix} r(x) \\ x \end{bmatrix}/\Delta_{\alpha \alpha} +_{r(x)} \sigma \circ d \left( \kappa^{-1} \circ \pi \left( \begin{bmatrix} r(x) \\ x \end{bmatrix}/\Delta_{\alpha \alpha} \right)   \right) .
\end{align*}
This shows $\phi(\sigma,\kappa)(d) = d^{(\sigma,\kappa)}$. By the evaluation for 2-cocycle above we see that $S \left( (\gamma,\beta), (\sigma,\kappa) \right) = h_{(\gamma,\beta)} +_{r(x)} h_{(\sigma,\kappa)}^{(\gamma,\beta)} -_{r(x)} h_{(\gamma \sigma,\beta \kappa)}$ which shows $[S]=0$.
\end{proof}

Theorem~\ref{thm:2} decomposes the kernel-preserving automorphisms of the extension $A_{T}(Q,A^{\alpha,\tau},\ast)$ in terms of the derivations and the compatible automorphisms of the datum. This is a decomposition into simpler objects since calculating in the extension requires the full algebraic structure of action and 2-cocycle terms while determining the compatible automorphisms requires just the action terms and the quotient algebra; that is, the algebraic structure of the semidirect product of the datum and the quotient algebra. In the case of groups \cite{wells}, the compatible automorphisms can be further decomposed into automorphisms of the normal subgroup of the kernel and of the quotient algebra which are connected by an analogue of the relation (C1). For general extensions of affine datum, computations of the compatible automorphisms are complicated by the fact that we no longer have a privileged kernel classes which decomposes further the automorphisms of the semidirect product; however, we will consider two general cases where some measure of simplification is possible.

We first consider the case of central extensions in varieties with a difference term. We rely on the characterization in \cite{wires1} of central extensions in such varieties. A ternary term $m$ is a \emph{difference term} for a variety $\mathcal V$ if for all algebras $A \in \mathcal V$ and congruences $\alpha \in \Con A$, the term satisfies 
\[
m(x,x,y) = y \quad \quad \text{and} \quad \quad m(x,y,y) \, [\alpha,\alpha] \, x \quad \quad \quad \quad \quad \quad (x,y, \in \alpha).
\]
It is immediate that the difference term interprets as a Mal'cev operation in any abelian algebra in the variety; in fact, the abelian algebras form a Mal'cev subvariety. We refer the reader to \cite[Thm 7.38]{bergman} for details of the close connection between abelian Mal'cev varieties and varieties of modules, but for the present purpose it suffices to be explicit about one part of that construction. Let $\mathcal A$ be an abelian Mal'cev variety in the signature $\tau$ with Mal'cev term $m(x,y,z)$. Consider the set of terms $R = \left\{ r(x,y) \in F_{\mathcal A}(x,y) : r(y,y) = y \right\}$ in the 2-generated free algebra in $\mathcal A$. Then $R=\left\langle R, +, - , \cdot, y, x \right\rangle$ is a unital ring under the definitions
\begin{align*}
r_{1} + r_{2} &:= m(r_{1},y,r_{2}) , &-r := m(y,r,y) &\quad \quad \text{and} &r_{1}(x,y) \cdot r_{2}(x,y) := r_{1}(r_{2}(x,y),y)
\end{align*}
where $y$ is the zero element for addition and $x$ is the identity for multiplication. Given an abelian algebra $A \in \mathcal A$ and choice of element $\mu \in A$, then $M(A,\mu)$ is defined as the $R$-module over the universe of $A$ where the abelian group operation is given by
\begin{align*}
a + b &:= m^{A}(a,\mu,b) &-a := m(\mu,a,\mu)
\end{align*} 
with zero element $\mu$ and the action of the ring $R$ is given by 
\[
r(x,y) \cdot a := r^{A}(a,\mu) 
\]
for $r(x,y) \in R$, $a \in M$. Given $f \in \tau$, define the terms
\begin{align*}
r_{f,i}(x,y) &:= m(f(y,\ldots,y,x,y,\ldots,y),f(y,\ldots,y),y) \\
d_{f}(x) &:= f(x,\ldots,x)
\end{align*}
for $1 \leq i \leq \ar f$. It follows that $r_{f,i} \in R$. The algebra $A$ and the $R$-module $M(A,\mu)$ are polynomially equivalent where the operations of $A$ can be represented in the module operations by
\begin{align}\label{eqn:idemmodule}
f^{A}(x_{1},\ldots,x_{n}) := \sum_{i=1}^{n} r_{f,i} \cdot x_{i} + d_{f}(\mu) .
\end{align}
It follows from Eq~\eqref{eqn:idemmodule} that if $\mu \in A$ is an idempotent element for the operations, then $A$ and $M(A,\mu)$ are term-equivalent since $d_{f}(\mu)=\mu$ the zero of the module structure.

\begin{proposition}\label{prop:centralcomp}
Let $A \in \mathcal V$ a variety with a difference term and $\pi: A \rightarrow Q$ a central extension realizing affine datum $(Q, A^{\alpha,\tau},\ast)$. Assume $Q$ has an idempotent element. Then the compatible automorphism 
\begin{align}\label{eq:508}
C(Q, A^{\alpha,\tau},\ast) \approx \Aut^{0} A(\alpha)/\Delta_{\alpha 1} \times \Aut Q \approx  \Aut M(A(\alpha)/\Delta_{\alpha 1}, \hat{\delta}) \times \Aut Q
\end{align} 
where $\Aut^{0} A(\alpha)/\Delta_{\alpha 1}$ is the group of automorphisms which fix $\hat{\delta}$ the $\Delta_{\alpha 1}$-class of the diagonal elements.
\end{proposition}
\begin{proof}
Let us focus on the first isomorphism in Eq~\eqref{eq:508} and begin by noting what the assumptions provide us. We are given that $Q = A/\alpha$. We write $\rho: A(\alpha)/\Delta_{\alpha \alpha} \rightarrow Q$ for the extension induced by $\pi$ of the associated semidirect product. Since $A$ realizes affine datum, if we fix an $\alpha$-trace $r: A \rightarrow A$, then every element in $A(\alpha)/\Delta_{\alpha \alpha}$ is uniquely represented in the form $\begin{bmatrix} r(a) \\ a \end{bmatrix}/\Delta_{\alpha \alpha}$ for $a \in A$. Let $l: Q \rightarrow A$ be the lifting such that $r = l \circ \pi$. Let $v \in Q$ be an idempotent element and choose $u \in \pi^{-1}(v)$. Since $A \in \mathcal V$ has a difference term and $\alpha = \ker \pi$ is central, we have by \cite[Lem 3.35]{wires1} an isomorphism for the semidirect product $A(\alpha)/\Delta_{\alpha \alpha} \approx A(\alpha)/\Delta_{\alpha 1} \times Q$ given by $\eta : \begin{bmatrix} a \\ b \end{bmatrix}/\Delta_{\alpha \alpha} \mapsto \left\langle \begin{bmatrix} a \\ b \end{bmatrix}/\Delta_{\alpha 1} \, ,\,  \pi(a) \right\rangle$ where the first-coordinate is given by the canonical epimorphism $\phi: A(\alpha)/\Delta_{\alpha \alpha} \rightarrow A(\alpha)/\Delta_{\alpha 1}$ for the congruence $\Delta_{\alpha 1}/\Delta_{\alpha \alpha}$. Since $\alpha$ is central, the diagonal elements of $A(\alpha)$ form a singleton $\Delta_{\alpha 1}$-class denoted by $\hat{\delta}$. Then $\Aut^{0} A(\alpha)/\Delta_{\alpha 1} = \{ \sigma \in \Aut A(\alpha)/\Delta_{\alpha 1}: \sigma ( \hat{\delta} ) = \hat{\delta} \}$.

Let us note how condition (C1) on the action relates to the algebra $A(\alpha)/\Delta_{\alpha 1}$. According to \cite[Lem 3.6(1a)]{wires1}, $A(\alpha)/\Delta_{\alpha 1}$ is an abelian algebra in which $\hat{\delta}$ is an idempotent element; thus, $A(\alpha)/\Delta_{\alpha 1}$ is term-equivalent to $M( A(\alpha)/\Delta_{\alpha 1} , \hat{\delta} )$. This means for operation symbol $f \in \tau$ with $n=\ar f$, there are ring elements $r_i \in R$ such that $f$ interprets as $f^{A(\alpha)/\Delta_{\alpha 1}}(x_1,\ldots,x_n) = r_{1} \cdot x_{1} + \cdots + r_{n} \cdot x_{1}$. Then by realization of the datum, we can calculate the action terms by 
\begin{align*}
\phi \circ a(f,i)(q_{1},\ldots,x,\ldots,q_{n}) &= \phi \circ F_{f} \big( \delta \circ l(q_{1}),\ldots,x,\ldots,\delta \circ l(q_{n}) \big) \\
&= \phi \circ f^{A(\alpha)/\Delta_{\alpha \alpha}} \big( \delta \circ l(q_{1}),\ldots,x,\ldots,\delta \circ l(q_{n}) \big) \\
&= r_{1} \cdot \hat{\delta} + \cdots + r_{i} \cdot \phi(x) + \cdots + r_{n} \cdot \hat{\delta} \\
&= r_{i} \cdot \phi(x)
\end{align*}
since each $\delta \circ l(q_{i})$ is a diagonal $\Delta_{\alpha \alpha}$-class; therefore, for complimentary automorphisms $(\sigma,\kappa)$ we have 
\begin{align}\label{eqn:200}
\phi \circ a(f,i)(\kappa(q_{1}),\ldots,\sigma(x),\ldots,\kappa(q_{n})) = r_{i} \cdot \phi(\sigma(x)).
\end{align}

Given a pair of compatible automorphisms $(\sigma,\kappa) \in C(Q, A^{\alpha,\tau},\ast)$, define $\hat{\sigma}: A(\alpha)/\Delta_{\alpha 1} \rightarrow A(\alpha)/\Delta_{\alpha 1}$ by $\hat{\sigma} \left(\begin{bmatrix} b \\ a \end{bmatrix}/\Delta_{\alpha 1} \right) := \phi \circ \sigma \left(\begin{bmatrix} r(u) \\ m(a,b,r(u)) \end{bmatrix}/\Delta_{\alpha \alpha} \right)$. Note by (C3) that $\sigma \left( \begin{bmatrix} b \\ b \end{bmatrix}/\Delta_{\alpha 1} \right)$ and $\sigma \left( \begin{bmatrix} r(u) \\ r(u) \end{bmatrix}/\Delta_{\alpha 1} \right)$ are both $\Delta_{\alpha \alpha}$-classes of a diagonal element. We observe
\begin{align}\label{eqn:15}
\begin{split}
\hat{\sigma} \circ \phi \left( \begin{bmatrix} b \\ a \end{bmatrix}/\Delta_{\alpha \alpha} \right) &= \hat{\sigma} \left(\begin{bmatrix} b \\ a \end{bmatrix}/\Delta_{\alpha 1} \right) \\
&= \phi \circ \sigma \left( \begin{bmatrix} r(u) \\ m(a,b,r(u)) \end{bmatrix}/\Delta_{\alpha \alpha} \right) \\
&= m \left( \phi \circ \sigma \left( \begin{bmatrix} b \\ a \end{bmatrix}/\Delta_{\alpha \alpha} \right) , \phi \circ \sigma \left( \begin{bmatrix} b \\ b \end{bmatrix}/\Delta_{\alpha \alpha} \right) , \phi \circ \sigma \left( \begin{bmatrix} r(u) \\ r(u) \end{bmatrix}/\Delta_{\alpha \alpha} \right) \right) \\
&= m\left( \phi \circ \sigma \left( \begin{bmatrix} b \\ a \end{bmatrix}/\Delta_{\alpha \alpha} \right), \hat{\delta}, \hat{\delta} \right) \\ 
&= \phi \circ \sigma \left( \begin{bmatrix} b \\ a \end{bmatrix}/\Delta_{\alpha \alpha} \right)
\end{split}
\end{align}
since $A(\alpha)/\Delta_{\alpha 1}$ is an abelian algebra. We can see from Eq~\eqref{eqn:15} and condition (C3) that $\hat{\sigma}(\hat{\delta}) = \hat{\delta}$. By condition (C2) we see that
\begin{align*}
\rho \circ \sigma \left( \begin{bmatrix} r(u) \\ m(a,b,r(u)) \end{bmatrix}/\Delta_{\alpha \alpha} \right) &= \kappa \circ \rho \left( \begin{bmatrix} r(u) \\ m(a,b,r(u)) \end{bmatrix}/\Delta_{\alpha \alpha} \right) \\
&= \kappa \circ m \left( \rho \left( \begin{bmatrix} b \\ a \end{bmatrix}/\Delta_{\alpha \alpha} \right), \rho \left( \begin{bmatrix} b \\ b \end{bmatrix}/\Delta_{\alpha \alpha} \right), \rho \left( \begin{bmatrix} r(u) \\ r(u) \end{bmatrix}/\Delta_{\alpha \alpha} \right)  \right) \\
&= \kappa \circ \rho \left( \begin{bmatrix} r(u) \\ r(u) \end{bmatrix}/\Delta_{\alpha \alpha} \right) \\
&= \kappa(v)
\end{align*}
since $\rho \left( \begin{bmatrix} b \\ a \end{bmatrix}/\Delta_{\alpha \alpha} \right) = \rho \left( \begin{bmatrix} b \\ b \end{bmatrix}/\Delta_{\alpha \alpha} \right)$. Using the isomorphism $\eta$ we can then represent 
\begin{align}\label{eqn:16}
\sigma \left( \begin{bmatrix} r(u) \\ m(a,b,r(u)) \end{bmatrix}/\Delta_{\alpha \alpha} \right) \longmapsto \left\langle \hat{\sigma} \left(\begin{bmatrix} b \\ a \end{bmatrix}/\Delta_{\alpha 1} \right) , \kappa(v) \right\rangle.
\end{align}
Since $v \in Q$ is idempotent, it follows that $\hat{\sigma}$ is a homomorphism. We also see from Eq~\eqref{eqn:200} condition (C2) collapses to the condition that $\hat{\sigma}$ respects the module structure on $A(\alpha)/\Delta_{\alpha \alpha}$ which is already guaranteed since it is a homomorphism which fixes the zero.

Surjectivity of $\sigma$ and $\phi$ guarantee by Eq~\eqref{eqn:15} that $\hat{\sigma}$ is also surjective. To show injectivity, assume $\hat{\sigma}\left( \begin{bmatrix} b \\ a \end{bmatrix}/\Delta_{\alpha 1} \right) = \hat{\sigma}\left( \begin{bmatrix} d \\ c \end{bmatrix}/\Delta_{\alpha 1} \right)$. Then by the representation in Eq~\eqref{eqn:16} and injectivity of $\sigma$ as an automorphism we conclude that $\begin{bmatrix} r(u) \\ m(a,b,r(u)) \end{bmatrix}/\Delta_{\alpha \alpha} = \begin{bmatrix} r(u) \\ m(c,d,r(u)) \end{bmatrix}/\Delta_{\alpha \alpha}$. Then passing to the quotient
\begin{align*}
\begin{bmatrix} b \\ a \end{bmatrix}/\Delta_{\alpha 1} = m \left( \begin{bmatrix} b \\ a \end{bmatrix}/\Delta_{\alpha 1}, \hat{\delta}, \hat{\delta} \right) &= m \left( \phi \left(\begin{bmatrix} b \\ a \end{bmatrix}/\Delta_{\alpha \alpha} \right), \phi \left(\begin{bmatrix} b \\ b \end{bmatrix}/\Delta_{\alpha \alpha} \right) , \phi \left(\begin{bmatrix} r(u) \\ r(u) \end{bmatrix}/\Delta_{\alpha \alpha} \right) \right) \\
&= \phi \left( \begin{bmatrix} r(u) \\ m(a,b,r(u)) \end{bmatrix}/\Delta_{\alpha \alpha} \right) \\
&= \phi \left( \begin{bmatrix} r(u) \\ m(c,d,r(u)) \end{bmatrix}/\Delta_{\alpha \alpha} \right) \\
&= \begin{bmatrix} d \\ c \end{bmatrix}/\Delta_{\alpha 1}; 
\end{align*}
altogether, $\hat{\sigma}$ is an automorphism.

For the last step, define $\psi: C(Q, A^{\alpha,\tau},\ast) \rightarrow \Aut^{0} A(\alpha)/\Delta_{\alpha 1} \times \Aut Q$ by $\psi(\sigma,\kappa) := (\hat{\sigma},\kappa)$. We can use Eq~\eqref{eqn:15} to show $\psi$ is a homomorphism by observing for all $x \in A(\alpha)/\Delta_{\alpha \alpha}$, 
\begin{align*}
\left( \hat{\gamma} \circ \hat{\sigma} \right) ( \phi(x) ) = \hat{\gamma} \circ \phi \circ \sigma(x) = \phi \circ \gamma \circ \sigma (x) = \widehat{\gamma \circ \sigma}(x) 
\end{align*}
which shows $\psi$ is a homomorphism.

We show surjectivity of $\phi$. Given $(\sigma,\kappa) \in \Aut^{0} A(\alpha)/\Delta_{\alpha 1} \times \Aut Q$ define $\lambda: A(\alpha)/\Delta_{\alpha \alpha} \rightarrow  A(\alpha)/\Delta_{\alpha \alpha}$ by the rule $\lambda \left( \begin{bmatrix} b \\ a \end{bmatrix}/\Delta_{\alpha \alpha} \right) = \eta^{-1} \left( \left\langle  \sigma \left( \begin{bmatrix} b \\ a \end{bmatrix}/\Delta_{\alpha 1} \right) , \kappa \circ \pi(b) \right\rangle  \right)$. It is straightforward to see that $\lambda$ is an automorphism and $(C2)$ holds. Since $\sigma$ fixes $\hat{\delta}$ we see that 
\begin{align*}
\lambda \left( \begin{bmatrix} a \\ a \end{bmatrix}/\Delta_{\alpha \alpha} \right) = \eta^{-1} \left( \left\langle  \sigma \left( \hat{\delta} \right) , \kappa \circ \pi(a) \right\rangle  \right) = \eta^{-1} \left( \left\langle \hat{\delta}  , \kappa \circ \pi(a) \right\rangle  \right) = \begin{bmatrix} l \circ \kappa \circ \pi(a) \\ l \circ \kappa \circ \pi(a) \end{bmatrix}/\Delta_{\alpha \alpha}
\end{align*}
which shows $\lambda$ satisfies condition (C3); therefore, $(\lambda,\kappa)$ is a complimentary pair. This establishes surjectivity of $\psi$. To show injectivity of $\psi$, suppose $\psi(\sigma,\kappa) = (\id,\id)$; thus, $\hat{\sigma} = \id$ and $\kappa = \id$. Then by (C2), we have $\rho \circ \sigma(x) = \kappa \circ \rho (x) = \rho(x)$. This implies $(\sigma(x),x) \in \hat{\alpha}/\Delta_{\alpha \alpha}$ for all $x \in A(\alpha)/\Delta_{\alpha \alpha}$. We also see using Eq~\eqref{eqn:15} that $\phi(x) = \hat{\sigma} \circ \phi \left( x \right) = \phi \circ \sigma \left( x \right)$ which implies $(\sigma(x),x) \in \Delta_{\alpha 1}/\Delta_{\alpha \alpha}$; thus, we see that $(\sigma(x),x) \in \hat{\alpha}/\Delta_{\alpha \alpha} \wedge \Delta_{\alpha 1}/\Delta_{\alpha \alpha}$. By \cite[Lem 3.5(3)]{wires1} we have $\Delta_{\alpha \alpha} = \Delta_{\alpha 1} \wedge \hat{\alpha}$ since $\alpha$ is central; therefore, $\sigma(x)=x$ for all $x \in A(\alpha)/\Delta_{\alpha \alpha}$.  We have shown $\psi$ is an isomorphism.

The second isomorphism in Eq~\eqref{eq:508} is given by the isomorphism $\Aut^{0} A(\alpha)/\Delta_{\alpha 1} \approx \Aut M(A(\alpha)/\Delta_{\alpha 1}, \hat{\delta})$ which follows from Eq~\eqref{eqn:idemmodule} and idempotency of $\hat{\delta} \in A(\alpha)/\Delta_{\alpha 1}$.
\end{proof}

\begin{corollary}\label{cor:800}
Let $A \in \mathcal V$ a variety with a difference term and and $\pi: A \rightarrow Q$ a central extension realizing affine datum $(Q,A^{\alpha,\tau},\ast)$. Assume $\mu \in A$ is idempotent and let $I_{\alpha}$ be the congruence class of $\alpha \in \Con A$ containing $\mu$. Then $C(Q,A^{\alpha,\tau},\ast) \approx \Aut M(I_{\alpha},\mu) \times \Aut Q $. 
\end{corollary}
\begin{proof}
This follows from Proposition~\ref{prop:centralcomp} when we observe that the map $\begin{bmatrix} a \\ b \end{bmatrix}/\Delta_{\alpha 1} \longmapsto m(b,a,\mu)$ shows $A(\alpha)/\Delta_{\alpha 1} \approx I_{\alpha}$ \cite[Sec 2]{wires1}.
\end{proof}

\begin{example}\label{ex:1}
We illustrate with a simple and familiar example. Consider an $R$-module $M$ with submodule $I \leq M$ and set $Q := M/I$. Let $\pi: M \rightarrow Q$ denote the canonical surjection. Since a module is an abelian algebra, the extension is central. Having fixed a lifting $l : Q \rightarrow M$ of $\pi$, then $I \rtimes_{T} Q$ is the module on the set $I \times Q$ with operations
\begin{itemize}

	\item $\left\langle a,x \right\rangle + \left\langle b,y \right\rangle := \left\langle a + b + T_{+}(x,y), x + y \right\rangle$,
	
	\item $r \cdot \left\langle a,x \right\rangle := \left\langle r \cdot a + T_{r}(x), x \right\rangle$,

\end{itemize}
where $T=\{T_{+}, T_{r} : r \in R\}$ is defined by
\begin{itemize}

	\item $T_{+}(x,y) = l(x) + l(y) - l(x+y)$ for $x,y \in Q$,

	\item $T_{r}(x) = r \cdot l(x) - l(r \cdot x)$ for $x \in Q, r \in R$.
	
\end{itemize}
If $\alpha_{I} = \{ (a,b): a-b \in I \}$ is the congruence determined by $I$, then $M(\alpha_{I})/\Delta_{\alpha_{I} 1} \approx I$ witnessed by the isomorphism $\begin{bmatrix} a \\ b \end{bmatrix}/\Delta_{\alpha_{I} \alpha_{I}} \mapsto b-a$ and $M \approx I \otimes^{T} Q$ given by $\begin{bmatrix} a \\ b \end{bmatrix}/\Delta_{\alpha_{I} \alpha_{I}} \mapsto \left\langle b-a, \pi(a) \right\rangle$ \cite[Sec 4]{wires1}. We saw in Eq~\eqref{eqn:200} that the action terms correspond to the module action in $I$; therefore, $\mathrm{Der}(Q,I) = \Hom_{R}(Q,I)$. Corollary~\ref{cor:800} yields the compatible automorphisms $C(Q,I) = \Aut I \times \Aut Q$.

We now consider the case of the direct sum $M = I \oplus Q$. Since the 2-cocycle $T=0$, we have $\ker W_{T} = C(Q,I) = \Aut I \times \Aut Q$ and so by Theorem~\ref{thm:2} we recover the semidirect decomposition
\begin{align}\label{eq:15}
\Aut_{I} M \approx \Hom_{R}(Q,I) \rtimes_{\gamma} (\Aut I \times \Aut Q)
\end{align}
of the group of nonsingular transformations which have $I$ as an invariant subspace. The action $\gamma$ is defined by the lifting $\hat{l}:\Aut I \times \Aut Q \rightarrow \Aut_{I} M$ of $\psi$ given by $\hat{l}(\sigma,\kappa)\left\langle a, x\right\rangle = \left\langle \sigma(a), \kappa(x) \right\rangle$. The identification of stabilizing automorphisms and derivations  $\phi \mapsto d_{\phi}$ is determined by $\phi \left\langle a,x \right\rangle = \left\langle a + d_{\phi}(x), x\right\rangle$. The action is then calculated by $(\sigma,\kappa) \ast \phi \left\langle a,x \right\rangle = \hat{l}(\sigma,\kappa) \circ \phi \circ \hat{l}(\sigma,\kappa)^{-1} \left\langle a,x \right\rangle = \left\langle a + d^{(\sigma,\kappa)}_{\phi}(x), x \right\rangle$ where $d^{(\sigma,\kappa)}_{\phi}(x) := \sigma \circ d_{\phi}(\kappa^{-1}(x))$.

We can also consider the matrix representation of automorphisms in $\phi \in \Aut_{I} M$ afforded by the direct sum $M = I \oplus Q$ in the form
\[
[\phi] = \begin{bmatrix} \sigma & D \\ 0 & \kappa \end{bmatrix}
\]
where $\sigma \in \Aut I, \kappa \in \Aut Q, D \in \Hom_{R}(Q,I)$. In terms of the standard coset-decomposition in groups, the isomorphism in Eq~(\ref{eq:15}) then takes the form
\begin{align*}
\Aut_{I} M  \ni \begin{bmatrix} \sigma & D \\ 0 & \kappa \end{bmatrix} &= \begin{bmatrix} \sigma & D \\ 0 & \kappa \end{bmatrix} \begin{bmatrix} \sigma^{-1} & 0 \\ 0 & \kappa^{-1} \end{bmatrix} \begin{bmatrix} \sigma & 0 \\ 0 & \kappa \end{bmatrix}  \\
&= \begin{bmatrix} I & D \circ \kappa^{-1} \\ 0 & I \end{bmatrix} \begin{bmatrix} \sigma & 0 \\ 0 & \kappa \end{bmatrix}  \longmapsto \left\langle D \circ \kappa^{-1}, (\sigma,\kappa) \right\rangle \in \Hom_{R}(Q,I) \rtimes_{\gamma} \left( \Aut I \times \Aut Q \right),
\end{align*}
and the group product of automorphisms is represented in the two ways by
\begin{align*}
\begin{bmatrix} \alpha \sigma & \alpha D + E \kappa \\ 0 & \beta \kappa \end{bmatrix} &= \begin{bmatrix} \alpha & E \\ 0 & \beta \end{bmatrix} \begin{bmatrix} \sigma & D \\ 0 & \kappa \end{bmatrix} \\
&\longmapsto \left\langle E\beta^{-1}, (\alpha,\beta) \right\rangle \left\langle D\kappa^{-1}, (\sigma,\kappa) \right\rangle = \left\langle E\beta^{-1} + \alpha D\kappa^{-1} \beta^{-1}, (\alpha \sigma, \beta \kappa) \right\rangle.
\end{align*}
\end{example}

For the next case, we assume $\alpha \in \Con A$ is only abelian but strengthen the hypothesis on the difference term variety and assume the extension has an idempotent. Extensions with abelian kernels can then be decomposed by the quotient algebra and a privileged congruence class which provides a translation of Theorem~\ref{thm:2} for these algebras. We briefly describe the facts of the representation but a more thorough development can be found in \cite[Sec 4]{wires1}.

We assume $A \in \mathcal V$ a variety with a weakly-associative difference term $m$ and $u \in A$ is an idempotent element. Let $\pi: A \rightarrow Q$ be an extension with $\alpha = \ker \pi$ abelian and choose a section $l: Q \rightarrow A$ of $\pi$ with $u = l \circ \pi(u)$. Write $Q = A/\alpha$. The extension induces affine datum $(Q,A^{\alpha,\tau},\ast)$ so that $A \approx A_{T}(Q,A^{\alpha,\tau},\ast)$. Define new action terms $\hat{\ast} = \{\hat{a}(f,i) : f \in \tau , i \in [\ar f] \}$ and 2-cocycle terms $S = \{S_{f} : f \in \tau \}$ in the following manner: for $f \in \tau$ with $n=\ar f$, define for $(\vec{x},\vec{a}) \in Q^{\ar f} \times I_{\alpha}^{\ar f}$
\begin{enumerate}

	\item $S_{f}(x_{1},\ldots,x_{n}) = m(f^{A}(l(x_{1}),\ldots,l(x_{n})),l(f^{Q}(x_{1},\ldots,x_{n})),u)$
	
	\item $\hat{a}(f,i)(\vec{x},\vec{a}) = m \big( f^{A}(l(x_{1}),\ldots,m(a_{i},u,l(x_{i})),\ldots,l(x_{n})), f^{A}(l(x_{1}),\ldots,l(x_{n})), u \big)$.

\end{enumerate}
The algebra $I_{\alpha} \rtimes_{\hat{\ast},S} Q$ is defined over the universe $I_{\alpha} \times Q$ with operations
\begin{align*}
f^{I_{\alpha} \rtimes_{\hat{\ast},S} Q}\left(\left\langle a_{1}, x_{1} \right\rangle,\ldots,\left\langle a_{1}, x_{1} \right\rangle \right) = \left\langle \sum_{i=1}^{n} \hat{a}(f,i)(\vec{x},\vec{a}) +_{u} S_{f}(\vec{x})  , f^{Q}(\vec{x}) \right\rangle
\end{align*}
where the sum is the iterated addition $m(x,u,x)= x +_{u} y$. The map $\psi : I_{\alpha} \rtimes_{\hat{\ast},S} Q \rightarrow A_{T}(Q,A^{\alpha,\tau},\ast)$ defined by 
\begin{align}\label{eqn:iso500}
\psi(a,x) = \begin{bmatrix} l(x) \\ m(a,u,l(x)) \end{bmatrix}/\Delta_{\alpha \alpha} 
\end{align}
is an isomorphism with inverse $\psi^{-1}\left( \begin{bmatrix} b \\ a \end{bmatrix}/\Delta_{\alpha \alpha} \right) = \left\langle m(a,b,u) , \pi(b) \right\rangle$. In terms of the map $\psi$, the two sets of action and 2-cocycle terms are related by  
\begin{enumerate}

	\item $\psi(\hat{a}(f,i)(\vec{x},\vec{a}),f^{Q}(\vec{x}) = a(f,i)(x_{1},\ldots,\psi(a_{i},x_{i}),\ldots,x_{n})$

	\item $\psi(S_{f}(\vec{x}),f^{Q}(\vec{x})) = T_{f}(\vec{x})$

\end{enumerate}
An element in an algebra is \emph{characteristic} if it is fixed by every automorphism. We will need to assume the preferred idempotent element is characteristic; of course, this is the case if the idempotent element happens to be a constant of the signature. The \emph{compatible automorphisms} $C(Q,I_{\alpha},\hat{\ast})$ is the set of all pairs $(\sigma,\kappa) \in \Aut I_{\alpha} \times \Aut Q$ such that
\begin{align}\label{eqn:compat200}
\sigma \circ \hat{a}(f,i)(\vec{x},\vec{a}) = \hat{a}(f,i)(\kappa(\vec{x}),\sigma(\vec{a}))
\end{align}
We use the isomorphism Eq~\eqref{eqn:iso500} to connect the definitions of compatible automorphisms for the two representations.

\begin{proposition}\label{prop:diffclass}
Let $A \in \mathcal V$ a variety with a weakly-associative difference term $m$ and $u \in A$ is a characteristic idempotent. If $\alpha \in \Con A$ is abelian and $Q = A/\alpha$, then $C(Q,A^{\alpha,\tau},\ast) \approx C(Q,I_{\alpha},\hat{\ast})$.
\end{proposition}
\begin{proof}
Fix a section $l: Q \rightarrow A$ for the canonical surjection $\pi : A \rightarrow Q$ associated to $\alpha$ such that $l \circ \pi(u)=u$. Then $l$ determines the two representation $A \approx A_{T}(Q,A^{\alpha,\tau},\ast) \approx I_{\alpha} \rtimes_{\hat{\ast},S} Q$. We utilize the previous discussion to relate together the actions $\ast$ and $\hat{\ast}$. Let us note that for $(\sigma,\kappa) \in C(Q,A^{\alpha,\tau},\ast)$, (C2) and (C3) imply that $\begin{bmatrix} u \\ u \end{bmatrix}/\Delta_{\alpha \alpha}$ must be fixed by $\sigma$ and so $\kappa(\pi(u)) = \pi(u)$.

Define $\Omega : C(Q,A^{\alpha,\tau},\ast) \rightarrow C(Q,I_{\alpha},\hat{\ast})$ by $\Omega(\sigma,\kappa):=(\hat{\sigma},\kappa)$ where the map $\hat{\sigma} : I_{\alpha} \rightarrow I_{\alpha}$ is defined by declaring $\hat{\sigma}(a) \in I_{\alpha}$ to be the unique value such that $\left\langle \hat{\sigma}(a) , \pi(u) \right\rangle = \psi^{-1} \circ \sigma \circ \psi (a,\pi(u))$. Since $\psi$ and $\sigma$ are isomorphisms, $\hat{\sigma}$ is well-defined. To see that $\hat{\sigma}$ is a homomorphism, observe that 
\begin{align*}
\left\langle \hat{\sigma} \circ f^{I_{\alpha}}(a_{1},\ldots,a_{n}) , \pi(u) \right\rangle &= \psi^{-1} \circ \sigma \circ \psi ( f^{I_{\alpha}}(a_{1},\ldots,a_{n}) , \pi(u) ) \\
&= \psi^{-1} \circ \sigma \circ \psi \big( f^{I_{\alpha} \rtimes_{\hat{\ast},S}}( \left\langle a_{1},\pi(u) \right\rangle ,\ldots, \left\langle a_{n},\pi(u) \right\rangle) \big) \\
&= f^{I_{\alpha} \rtimes_{\hat{\ast},S}} \big( \psi^{-1} \circ \sigma \circ \psi(a_{1},\pi(u)),\ldots, \psi^{-1} \circ \sigma \circ \psi(a_{n},\pi(u)) \big) \\
&= f^{I_{\alpha} \rtimes_{\hat{\ast},S}} \big( \left\langle \hat{\sigma}(a_{1}), \pi(u) \right\rangle,\ldots,\left\langle \hat{\sigma}(a_{1}), \pi(u) \right\rangle \big) \\
&= \left\langle f^{I_{\alpha}} \big( \hat{\sigma}(a_{1}),\ldots,\hat{\sigma}(a_{n}) \big) , \pi(u) \right\rangle
\end{align*}
which yields $\hat{\sigma} \circ f^{I_{\alpha}}(a_{1},\ldots,a_{n}) = f^{I_{\alpha}} \big( \hat{\sigma}(a_{1}),\ldots,\hat{\sigma}(a_{n}) \big)$. Since $\psi$ and $\sigma$ are bijective, it is easy to see that $\hat{\sigma}$ is too. The next step is to verify Eq~\eqref{eqn:compat200} for the pair $(\hat{\sigma},\kappa)$ whenever $(\sigma,\kappa) \in C(Q,A^{\alpha,\tau},\ast)$. If we set $\hat{a}(f,i)(x_{1},\ldots,a,\ldots,x_{n}) = \omega$ and $f^{Q}(x_{1},\ldots,\pi(u),\ldots,x_{n}) = \xi$, then 
\begin{align*}
\psi \big( \hat{a}(f,i) (\kappa(x_{1}),\ldots, &\hat{\sigma}(a),\ldots,\kappa(x_{n}) ) , f^{Q}(\kappa(x_{1}),\ldots,\pi(u),\ldots,\kappa(x_{n})) \big) \\
&= a(f,i)(\kappa(x_{1}),\ldots, \psi( \hat{\sigma}(a),\pi(u)),\ldots,\kappa(x_{n})) \\
&= a(f,i)(\kappa(x_{1}),\ldots,\sigma \circ \psi(a,\pi(u)),\ldots,\kappa(x_{n})) \\
&= \sigma \circ a(f,i)(x_{1},\ldots,\psi(a,\pi(u)),\ldots,x_{n}) \\
&= \sigma \left( \begin{bmatrix} l(\xi) \\ m(\omega,u,l(\xi)) \end{bmatrix}/\Delta_{\alpha \alpha} \right) \\
&=  m \left( \sigma \left( \begin{bmatrix} u \\ \omega \end{bmatrix}/\Delta_{\alpha \alpha} \right), \sigma \left( \begin{bmatrix} u \\ u \end{bmatrix}/\Delta_{\alpha \alpha} \right) , \sigma \left( \begin{bmatrix} l(\xi) \\ l(\xi) \end{bmatrix}/\Delta_{\alpha \alpha} \right) \right)  \\
&= m \left( \sigma \left( \begin{bmatrix} u \\ \omega \end{bmatrix}/\Delta_{\alpha \alpha} \right),  \begin{bmatrix} u \\ u \end{bmatrix}/\Delta_{\alpha \alpha} , \begin{bmatrix} l \circ \kappa (\xi) \\ l \circ \kappa (\xi) \end{bmatrix}/\Delta_{\alpha \alpha} \right) \\
&= m \big( \psi( \hat{\sigma}(\omega) , \pi(u)), \psi(u,\pi(u)) , \psi(u,\kappa(\xi)) \big)  \\
&= \psi \Big( m^{I_{\alpha} \rtimes_{\hat{\ast},S} Q} \big( \left\langle \hat{\sigma}(\omega), \pi(u) \right\rangle , \left\langle u , \pi(u) \right\rangle , \left\langle u , \kappa(\xi) \right\rangle  \big) \Big)  \\
&= \psi( \hat{\sigma}(\omega) , \kappa(\xi) )  \\
&= \psi \big( \hat{\sigma}(\omega) , f^{Q}(\kappa(x_{1}),\ldots,\pi(u),\ldots,\kappa(x_{n})) \big) .
\end{align*}
Identifying the first coordinates yield Eq~\eqref{eqn:compat200}.

To see that $\Omega$ is a homomorphism it suffices to observe that $\left\langle \widehat{\gamma \circ \sigma}(a) , \pi(u) \right\rangle = \psi^{-1} \circ \gamma \circ \sigma \circ \psi (a,\pi(u)) = \psi^{-1} \circ \gamma \circ \psi \circ \psi^{-1} \sigma \circ \psi (a,\pi(u)) = \psi^{-1} \circ \gamma \circ \psi (\hat{\sigma}(a), \pi(u)) = \left\langle \hat{\gamma} \circ \hat{\sigma}(a), \pi(u) \right\rangle$. For injectivity of $\Omega$, assume $(\hat{\sigma},\kappa) = (\id,\id)$. Then since $(\sigma,\id) \in C(Q,A^{\alpha,\tau},\ast)$, (C2) and (C3) implies $\sigma$ fixes each diagonal class.  
Then $\hat{\sigma} = \id$ implies 
\begin{align*}
m \left( \begin{bmatrix} l(x) \\ a \end{bmatrix}/\Delta_{\alpha \alpha}, \begin{bmatrix} l(x) \\ l(x) \end{bmatrix}/\Delta_{\alpha \alpha}, \begin{bmatrix} u \\ u \end{bmatrix}/\Delta_{\alpha \alpha} \right) &= \begin{bmatrix} u \\ m(a,l(x),u) \end{bmatrix}/\Delta_{\alpha \alpha} \\
&= \psi \big( m(a,l(x),u) , \pi(u) \big) \\
&= \psi \big( \hat{\sigma}(m(a,l(x),u)) , \pi(u) \big)  \\
&= \sigma \left( \begin{bmatrix} u \\ m(a,l(x),u) \end{bmatrix}/\Delta_{\alpha \alpha} \right)  \\
&= m \left( \sigma \left( \begin{bmatrix} l(x) \\ a \end{bmatrix}/\Delta_{\alpha \alpha} \right), \sigma \left( \begin{bmatrix} l(x) \\ l(x) \end{bmatrix}/\Delta_{\alpha \alpha} \right), \sigma \left( \begin{bmatrix} u \\ u \end{bmatrix}/\Delta_{\alpha \alpha} \right) \right)  \\
&= m \left( \sigma \left( \begin{bmatrix} l(x) \\ a \end{bmatrix}/\Delta_{\alpha \alpha} \right), \begin{bmatrix} l(x) \\ l(x) \end{bmatrix}/\Delta_{\alpha \alpha} ,  \begin{bmatrix} u \\ u \end{bmatrix}/\Delta_{\alpha \alpha}  \right) .
\end{align*}
Then $\mathcal V \vDash x = m(m(x,y,z),z,y)$ applied to the above equation yields $\sigma \left( \begin{bmatrix} l(x) \\ a \end{bmatrix}/\Delta_{\alpha \alpha} \right) = \begin{bmatrix} l(x) \\ a \end{bmatrix}/\Delta_{\alpha \alpha}$; that is, $\sigma = \id$.

The last step is to verify surjectivity of $\Omega$. Take $(\sigma,\kappa) \in C(Q,I_{\alpha},\hat{\ast})$ and define $\lambda : A(\alpha)/\Delta_{\alpha \alpha} \rightarrow A(\alpha)/\Delta_{\alpha \alpha}$ by the rule 
\begin{align}\label{eqn:lambda500}
\lambda\left( \begin{bmatrix} l(x) \\ a \end{bmatrix}/\Delta_{\alpha \alpha} \right) := \psi(\sigma \circ m(a,l(x),u),\kappa(x)).
\end{align}
Let us observe that 
\begin{align*}
\left\langle \hat{\lambda}(a) , \pi(u) \right\rangle = \psi^{-1} \circ \lambda \circ \psi(a,\pi(u)) &=  \psi^{-1} \circ \lambda \left( \begin{bmatrix} u \\ a \end{bmatrix}/\Delta_{\alpha \alpha} \right) \displaybreak[0]\\
&= \psi^{-1} \circ \psi ( \sigma(m(a,u,u)),\pi(u)) \displaybreak[0]\\
&= \left\langle \sigma(a), \pi(u) \right\rangle 
\end{align*}
and so $\hat{\lambda} = \sigma$. It remains to show $\lambda$ is an automorphism and $(\lambda,\kappa) \in C(Q,A^{\alpha,\tau},\ast)$. 

For injectivity, note $\lambda\left( \begin{bmatrix} l(x) \\ a \end{bmatrix}/\Delta_{\alpha \alpha} \right) = \lambda\left( \begin{bmatrix} l(y) \\ b \end{bmatrix}/\Delta_{\alpha \alpha} \right)$ implies $\sigma(m(a,l(x),u)) = \sigma(m(b,l(y),u))$ and $\kappa(x) = \kappa(y)$. Then $x=y$ and $m(a,l(x),u) = m(b,l(y),u)$ implies $a=b$. For surjectivity, we use Eq~\eqref{eqn:lambda500} to solve $\lambda \left( \begin{bmatrix} l(x) \\ a \end{bmatrix}/\Delta_{\alpha \alpha} \right) = \begin{bmatrix} l(y) \\ b \end{bmatrix}/\Delta_{\alpha \alpha}$. Since $\psi, \sigma,\kappa$ are isomorphisms, this can be accomplished by $a = m \big( \sigma^{-1} \circ m(b,l(y),u), u, l \circ \kappa^{-1}(y) \big)$ and $x = \kappa^{-1}(y)$. To see that $\lambda$ is a homomorphism, take $\vec{a} \in A^{\ar f}$ where $x_{i} = \pi(a_{i})$ and set $p_{i}=f^{A}(l(x_{1}),\ldots,a_{i},\ldots,l(x_{n}))$ and $w = f^{A}(l(x_{1}),\ldots,l(x_{n}))$ so that we can write $a(f,i)(x_{1},\ldots,\begin{bmatrix} l(x_{i}) \\ a_{i} \end{bmatrix}/\Delta_{\alpha \alpha},\ldots,x_{n}) = \begin{bmatrix} w \\ p_{i} \end{bmatrix}/\Delta_{\alpha \alpha}$. Then using the fact that $(\sigma,\kappa) \in C(Q,I_{\alpha},\hat{\ast})$ we have 
\begin{align*}
\sigma \circ m(p_{i},w,u) &= \sigma \circ m(f^{A}(l(x_{1}),\ldots,a_{i},\ldots,l(x_{n})),f^{A}(l(x_{1}),\ldots,l(x_{n})),u) \displaybreak[0]\\
&= \sigma \circ m(f^{A}(l(x_{1}),\ldots,m(m(a_{i},l(x_{i}),u),u,l(x_{i})),\ldots,l(x_{n})),f^{A}(l(x_{1}),\ldots,l(x_{n})),u) \displaybreak[0]\\
&= \sigma \circ \hat{a}(f,i)(\vec{x},(m(a_{1},l(x_{1}),u),\ldots,m(a_{n},l(x_{n}),u))) \displaybreak[0]\\
&= \hat{a}(f,i) \big(\kappa(\vec{x}),(\sigma \circ m(a_{1},l(x_{1}),u),\ldots,\sigma \circ m(a_{n},l(x_{n}),u)) \big) .
\end{align*}
We then calculate in the semidirect product 
\begin{align*}
\lambda \circ F_{f} &\left( \begin{bmatrix} l(x_{1}) \\ a_{1} \end{bmatrix}/\Delta_{\alpha \alpha} ,\ldots, \begin{bmatrix} l(x_{n}) \\ a_{n} \end{bmatrix}/\Delta_{\alpha \alpha} \right) \displaybreak[0]\\
&= \lambda \left( \sum_{i=1}^{n} a(f,i)(x_{1},\ldots,\begin{bmatrix} l(x_{i}) \\ a_{i} \end{bmatrix}/\Delta_{\alpha \alpha},\ldots,x_{n}) \right) \displaybreak[0]\\
&= \lambda \left( \begin{bmatrix} w \\ p_{1} +_{w} \cdots +_{w} p_{n} \end{bmatrix}/\Delta_{\alpha \alpha} \right) \displaybreak[0]\\ 
&= \psi \Big( \sigma \circ m(p_{1} +_{w} \cdots +_{w} p_{n}, w, u ) , \kappa(f^{Q}(\vec{x})) \Big) \displaybreak[0]\\ 
&= \psi \Big( \sigma( m(p_{1},w,u) +_{u} \cdots +_{u} m(p_{n},w,u) ) , \kappa(f^{Q}(\vec{x})) \Big) \displaybreak[0]\\ 
&= \psi \Big( \sum_{i=1}^{n} \sigma \circ m(p_{i},w,u) , \kappa(f^{Q}(\vec{x})) \Big) \displaybreak[0]\\ 
&= \psi \left( \sum_{i=1}^{n} \hat{a}(f,i) \big(\kappa(\vec{x}),(\sigma \circ m(a_{1},l(x_{1}),u),\ldots,\sigma \circ m(a_{n},l(x_{n}),u)) \big) , f^{Q}(\kappa(\vec{x})) \right) \displaybreak[0]\\ 
&= \psi \circ f^{I_{\alpha} \rtimes_{\hat{\ast}} Q} \Big( \left\langle \sigma \circ m(a_{1},l(x_{1}),u) , \kappa(x_{1}) \right\rangle ,\ldots,\left\langle \sigma \circ m(a_{n},l(x_{n}),u) , \kappa(x_{n}) \right\rangle \Big) \displaybreak[0]\\ 
&= F_{f} \Big( \psi(\sigma \circ m(a_{1},l(x_{1}),u),\kappa(x_{1})),\ldots,\psi(\sigma \circ m(a_{n},l(x_{n}),u),\kappa(x_{n})) \Big) \displaybreak[0]\\
&= F_{f} \Big( \lambda \left( \begin{bmatrix} l(x_{1}) \\ a_{1} \end{bmatrix}/\Delta_{\alpha \alpha} \right),\ldots, \lambda \left( \begin{bmatrix} l(x_{n}) \\ a_{n} \end{bmatrix}/\Delta_{\alpha \alpha} \right) \Big) .
\end{align*}
Finally, we see that 
\begin{align*}
\lambda \left( \begin{bmatrix} l(x) \\ l(x) \end{bmatrix}/\Delta_{\alpha \alpha} \right) &= \psi( \sigma \circ m(l(x),l(x),u) , k(x)) \displaybreak[0]\\
&= \psi(\sigma(u) , k(x)) = \psi(u,\kappa(x)) = \begin{bmatrix} l(\kappa(x)) \\ m(u,u,l(\kappa(x))) \end{bmatrix}/\Delta_{\alpha \alpha} = \begin{bmatrix} l(\kappa(x)) \\ l(\kappa(x)) \end{bmatrix}/\Delta_{\alpha \alpha} 
\end{align*}
which yields (C3) from which (C2) follows. Condition (C1) can be obtained by the homomorphism property
\begin{align*}
\lambda \circ a(f,i) &\big( x_{1},\ldots, \begin{bmatrix} l(y) \\ a \end{bmatrix}/\Delta_{\alpha \alpha},\ldots,x_{n} \big) \\
&= \lambda \circ F_{f} \left(  \begin{bmatrix} l(x_{1}) \\ l(x_{1}) \end{bmatrix}/\Delta_{\alpha \alpha},\ldots, \begin{bmatrix} l(y) \\ a \end{bmatrix}/\Delta_{\alpha \alpha} ,\ldots,\begin{bmatrix} l(x_{n}) \\ l(x_{n}) \end{bmatrix}/\Delta_{\alpha \alpha} \right) \\
&= F_{f} \left(  \lambda \left( \begin{bmatrix} l(x_{1}) \\ l(x_{1}) \end{bmatrix}/\Delta_{\alpha \alpha} \right),\ldots, \lambda \left( \begin{bmatrix} l(y) \\ a \end{bmatrix}/\Delta_{\alpha \alpha} \right) ,\ldots, \lambda \left( \begin{bmatrix} l(x_{n}) \\ l(x_{n}) \end{bmatrix}/\Delta_{\alpha \alpha} \right) \right) \\
&= F_{f} \left(  \begin{bmatrix} l(\kappa(x_{1})) \\ l(\kappa(x_{1})) \end{bmatrix}/\Delta_{\alpha \alpha},\ldots, \lambda \left( \begin{bmatrix} l(y) \\ a \end{bmatrix}/\Delta_{\alpha \alpha} \right),\ldots,\begin{bmatrix} l(\kappa(x_{n})) \\ l(\kappa(x_{n})) \end{bmatrix}/\Delta_{\alpha \alpha} \right) \\
&= a(f,i) \big( \kappa(x_{1}),\ldots, \lambda \left( \begin{bmatrix} l(y) \\ a \end{bmatrix}/\Delta_{\alpha \alpha} \right) ,\ldots,\kappa(x_{n}) \big) .
\end{align*}
\end{proof}

\begin{example}
Let $A$ be a group with multiple operators (or $\Omega$-group \cite{higgins}) and 1 the neutral group element with $f(1,\ldots,1)=1$ for all $f \in \Omega$. The difference term is given by the group term $m(x,y,z) = x \cdot y^{-1} \cdot z$ which clearly satisfies $x = m(m(x,y,z),z,y)$ by group properties; in addition, any homomorphism between $\Omega$-groups must also preserve the group structure which implies $1$ is a characteristic idempotent. For an abelian congruence $\alpha \in \Con A$ with $Q = A/\alpha$ and choice of section $l:Q \rightarrow A$ with $l(1)=1$, we have $A \approx I_{\alpha} \rtimes_{\ast,T} Q$ where
\begin{itemize}
	
	\item $a(\cdot,1)((x,y),(a,b)) = (a \cdot l(x) \cdot l(y)) \cdot (l(x) \cdot l(y))^{-1} = a$;
	
	\item $a(\cdot,2)((x,y),(a,b)) = (l(x) \cdot b \cdot l(y)) \cdot (l(x) \cdot l(y))^{-1} = b^{l(x)}$;
	
	\item $a(f,i)(\vec{x},\vec{a}) = f^{A}(l(x_{1}),\ldots,a_{i}\cdot l(x_{i}),\ldots,l(x_{n})) \cdot f^{A}(l(x_{1}),\ldots,l(x_{n}))^{-1}$  $(f \in \Omega)$;
	
	\item $T_{\cdot}(x,y) = l(x) \cdot l(y) \cdot l(x \cdot y)^{-1}$;
	
	\item $T_{f}(x_{1},\ldots,x_{n}) = f^{A}(l(x_{1}),\ldots,l(x_{n})) \cdot l(f^{Q}(x_{1},\ldots,x_{n}))^{-1}$ \ $(f \in \Omega)$ .

\end{itemize}

The compatible automorphisms consist of all pairs $(\sigma,\kappa) \in \Aut I_{\alpha} \times \Aut Q$ such that
\begin{itemize}

	\item $\sigma(a^{l(x)}) = \sigma(a)^{l\circ \kappa(x)}$
	
	\item $\sigma \big( f(l(x_{1}),\ldots,a_{i}\cdot l(x_{i}),\ldots,l(x_{n})) \cdot f(l(x_{1}),\ldots,l(x_{n}))^{-1} \big) = f(l \circ \kappa(x_{1}),\ldots,\sigma(a) \cdot l \circ \kappa(x_{i}),\ldots,l\circ \kappa(x_{n})) \cdot f(l \circ \kappa(x_{1}),\ldots,l \circ \kappa(x_{n}))^{-1}$ .

\end{itemize}

A special case of $\Omega$-groups occurs when $A$ is an $R$-module with multilinear operations indexed by $F$. The representation $A \approx I_{\alpha} \rtimes_{\ast,T} Q$ is constructed from 
\begin{itemize}

	\item $a(f,i)(\vec{x},\vec{a}) = f^{A}(l(x_{1}),\ldots,a_{i} + l(x_{i}),\ldots,l(x_{n})) - f^{A}(l(x_{1}),\ldots,l(x_{n})) = f^{A}(l(x_{1}),\ldots,a_{i},\ldots,l(x_{n}))$
	
	\item $T_{+}(x,y) = l(x) + l(y) - l(x+y)$;
	
	\item $T_{r}(x) = r \cdot l(x) - l(r \cdot x)$ \ $(r \in R)$;
	
	\item $T_{f}(x_{1},\ldots,x_{n}) = f^{A}(l(x_{1}),\ldots,l(x_{n})) - l(f^{Q}(x_{1},\ldots,x_{n}))$

\end{itemize}
since the action terms associated to the group operation are trivial by commutativity and multilinearity simplifies the actions associated to the operations in $F$. This agrees with \cite{multiext} when restricted to abelian ideals. The compatible automorphisms consist of all pairs $(\sigma,\kappa) \in \Aut I_{\alpha} \times \Aut Q$ such that
\begin{align*}
\sigma \circ f^{A}(l(x_{1}),\ldots,a_{i},\ldots,l(x_{n})) = f^{A}(l\circ \kappa(x_{1}),\ldots,\sigma(a_{i}),\ldots,l\circ \kappa(x_{n}))
\end{align*}
for each $f \in F$ which is at least binary. 
\end{example}

\begin{proof}(of Corollary~\ref{cor:1})
Since the idempotent $u$ is preserved by any automorphism, we see that an automorphism $\phi$ respects $\alpha$ if and only if $\phi(I_{\alpha}) = I_{\alpha}$; that is, $\Aut_{\alpha} A = \Aut_{I_{\alpha}} A$. The isomorphism $I_{\alpha} \rtimes_{\hat{\ast},T'} Q \approx A_{T}(Q,A^{\alpha,\tau},\ast)$ from Eq~\eqref{eqn:iso500} yields isomorphisms of cohomology \cite{wires1}
\begin{align*}
H^{2}_{\mathcal V}(Q,I_{\alpha},\ast) &\approx H^{2}_{\mathcal V}(Q,A^{\alpha,\tau},\ast') &\text{ and } & & \mathrm{Der}(Q,I_{\alpha},\ast) &\approx \mathrm{Der}(Q,A^{\alpha,\tau},\ast').
\end{align*}
given by
\begin{align*}
\psi(T'_{f}(\vec{x}),f^{Q}(\vec{x})) &= T_{f}(\vec{x}) &\text{ and } & & \psi(h'(x),x) = h(x) \text{ for } h' \in \mathrm{Der}(Q,I_{\alpha},\ast).
\end{align*}
Proposition~\ref{prop:diffclass} gives the isomorphism $\Omega : C(Q,A^{\alpha,\tau},\ast) \rightarrow C(Q,I_{\alpha},\hat{\ast})$. The exact sequence in Eq~\eqref{eqn:idemwells} is then reconstructed from the diagram
\[
\begin{tikzcd}
	1 & {\mathrm{Der}(Q,A^{\alpha,\tau},\ast)} & {\mathrm{Aut}_{\alpha} A } & {C(Q,A^{\alpha,\tau},\ast)} & {H^{2}_{\mathcal V}(Q,A^{\alpha,\tau},\ast)} & 1 \\
	1 & {\mathrm{Der}(Q,I_{\alpha},\hat{\ast})} & {\mathrm{Aut}_{I_{\alpha}} A } & {C(Q,I_{\alpha},\hat{\alpha})} & {H^{2}_{\mathcal V}(Q,I_{\alpha},\hat{\ast})} & 1
	\arrow[from=1-1, to=1-2]
	\arrow[from=1-2, to=1-3]
	\arrow["\approx", from=1-2, to=2-2]
	\arrow["\psi", from=1-3, to=1-4]
	\arrow["{=}", no head, from=1-3, to=2-3]
	\arrow["{W_{T}}", from=1-4, to=1-5]
	\arrow["\Omega", from=1-4, to=2-4]
	\arrow[from=1-5, to=1-6]
	\arrow["\approx", from=1-5, to=2-5]
	\arrow[from=2-1, to=2-2]
	\arrow[from=2-2, to=2-3]
	\arrow[from=2-3, to=2-4]
	\arrow[from=2-4, to=2-5]
	\arrow[from=2-5, to=2-6]
\end{tikzcd}
\]
\end{proof}

If we identify $A$ with its representation $I_{\alpha} \rtimes_{\ast,T} Q$, then Corollary~\ref{cor:1} provides the group extension $\Aut_{I_{\alpha}} A \stackrel{\psi}{\longrightarrow} \ker W_{T}$ which has a section $l_{T}$ which is determined by $l_{T}(\sigma,\kappa)(a,x) = \left\langle \sigma(a) +_{u} h \circ \kappa(x), \kappa(x) \right\rangle$ where $h: Q \rightarrow I_{\alpha}$ witnesses $T \sim T^{(\sigma,\kappa)}$. In terms of this extension, any automorphism $\phi \in \Aut_{I_{\alpha}} A$ is represented by a triple $\phi = (d_{\gamma},\sigma,\kappa) \in \mathrm{Der}(Q,I_{\alpha},\ast) \times \Aut \, I_{\alpha} \times \Aut \, Q$ such that 
\begin{itemize}

	\item $d_{\gamma}$ is the derivation corresponding to stabilizing automorphism $\gamma$ where $\gamma(a,x) = \left\langle a +_{u} d_{\gamma}(x), x \right\rangle$, 

	\item $\sigma \circ a(f,i)(\vec{x},\vec{a}) = a(f,i)(\kappa(\vec{x}),\sigma(\vec{a}))$, and 
	
	\item there exists $h: Q \rightarrow I_{\alpha}$ such that $T_{f}(\vec{x}) - \sigma \circ T_{f}(\kappa^{-1}(\vec{x})) = \sum_{i=1}^{\ar f} a(f,i)(\vec{x},h(\vec{x})) - h(f^{Q}(\vec{x}))$ for each $f \in \tau$.

\end{itemize}
The action of the automorphism on $A$ is reconstructed by 
\begin{align*}
\phi(a,x) = \gamma \circ l_{T}(\sigma,\kappa)(a,x) = \left\langle \sigma(a) +_{u} h \circ \kappa(x) +_{u} d_{\gamma} \circ \kappa(x) , \kappa(x) \right\rangle.
\end{align*}

\vspace{0.5cm}


\section{Demonstration for Theorem~\ref{thm:4}}\label{section:4}
\vspace{0.3cm}

Fix a variety $\mathcal V$ in the signature $\tau$ of $R$-modules expanded by multilinear operations named by $F$. Let $A \in \mathcal V$ and $\pi : A \rightarrow Q$ a surjective homomorphism with $\ker \pi$ determined by the ideal $I \triangleleft A$. Take $\phi \in \Aut _{I} A$. For any lifting $l: Q \rightarrow A$ of $\pi$, we can verify that $\phi_{l} := \pi \circ \phi \circ l: Q \rightarrow Q$ is an automorphism of $Q$. The restriction of $\phi$ to $I$ is denoted by $\phi|_{I} : I \rightarrow I$. As in Section~\ref{section:2}, it follows that $\psi: \Aut_{I} \rightarrow \Aut I \times \Aut Q$ given by $\psi(\phi) := (\phi|_{I},\phi_l)$ is a group homomorphism independent of the lifting $l$.

For nonabelian extensions, the actions terms are usually folded into the notion of a 2-cocycle (or factor system) and are affected by the equivalence determined by 2-coboundaries. The compatible automorphisms of the datum $(Q,I)$ is then given by the full direct product $\Aut I \times \Aut Q$, but the previous ``compatibility'' notion is incorporated into the action of the full product on cohomology classes $H^{2}_{\mathcal V}(Q,I)$ by the addition of a new twisting of the action terms in a 2-cocycle. The action of $\Aut I \times \Aut Q$ on the $\mathcal V$-compatible 2-cocycles $Z^{2}_{\mathcal V}(Q,I)$ is given by the following: for $(\sigma,\kappa) \in \Aut I \times \Aut Q$ and 2-cocycle $T \in  Z^{2}_{\mathcal V}(Q,I)$, define the 2-cocycle 
\[
T^{(\sigma,\kappa)} = \left\{ T^{(\sigma,\kappa)}_{+},T^{(\sigma,\kappa)}_{r},T^{(\sigma,\kappa)}_{f},a(f,s)^{(\sigma,\kappa)}: r \in R, f \in F, s \in [ \ar f]^{\ast} \right\}
\]
by the rules
\begin{align*}
T^{(\sigma,\kappa)}_{+}(x,y) &:= \sigma \circ T_{+}(\kappa^{-1}(x),\kappa^{-1}(y)) \\
T^{(\sigma,\kappa)}_{r}(x) &:= \sigma \circ T_{r}(\kappa^{-1}(x)) \\ 
T^{(\sigma,\kappa)}_{f}(\vec{x}) &:= \sigma \circ T_{f}(\kappa^{-1}(\vec{x})) \\
a(f,s)^{(\sigma,\kappa)}(\vec{x},\vec{a}) &:= \sigma \circ a(f,s)(\kappa^{-1}(\vec{x}),\sigma^{-1}(\vec{a})) .
\end{align*}
Given $t = s \in \Id \mathcal V$, because $T$ is a $\mathcal V$-compatible 2-cocycle we have the corresponding identity $t^{\partial,T} = s^{\partial,T}$ in the signature of the 2-cocycle such that the multi-sorted structure $\left\langle I \cup Q, \tau^{I}, \tau^{Q}, T \right\rangle \vDash t^{\partial,T} = s^{\partial,T}$. By applying $\sigma$ and $\kappa^{-1}$ to the interpretations of the terms $t^{\partial,T}$ and $s^{\partial,T}$ in a manner similar to that outlined in Section~\ref{section:2}, we would be able to conclude that $\left\langle I \cup Q, \tau^{I}, \tau^{Q}, T^{(\sigma,\kappa)} \right\rangle \vDash t^{\partial,T^{(\sigma,\kappa)}} = s^{\partial,T^{(\sigma,\kappa)}}$; therefore, $T^{(\sigma,\kappa)} \in Z^{2}_{\mathcal V}(Q,I)$. An argument similar to that given in Section~\ref{section:2} again shows $[T]^{(\sigma,\kappa)}:=[T^{(\sigma,\kappa)}]$ yields a well-defined action.

Define $W_{T}(\sigma,\kappa) := [T - T^{(\sigma,\kappa)}]$ for $(\sigma,\kappa) \in \Aut I \times \Aut Q$. For $[T] \in H^{2}_{\mathcal V}(Q,I)$ and $(\sigma,\kappa) \in \Aut I \times \Aut Q$, it may be that $[T - T^{(\sigma,\kappa)}] \not\in H^{2}_{\mathcal V}(Q,I)$ since the nonabelian cohomology classes may no longer be closed under the natural addition induced by $I$. One way to remedy this defect occurs in the original group case, where the nonabelian cohomology classes realizing a fixed action can be identified with the cohomology classes realizing a related action on the center of the kernel. This identification is then incorporated into the definition of the Wells derivation \cite{wells,autgrp}; unfortunately, an analogue of this comparison for our general varieties is not established in \cite{wires1} and so we resort to the convenient stopgap of taking the free abelian group generated by the cohomology classes as a formal codomain.

\begin{proof} (proof of Theorem \ref{thm:4})
As in the proof of Theorem~\ref{thm:2}, we see that the kernel of $\psi$ is given by $\ker \psi = \left\{ \phi \in \Aut_{I} A : \phi|_{I} = \id, \pi = \pi \circ \phi \right\} \approx \mathrm{Der}(Q,I)$.

We now demonstrate exactness at $\Aut {I} \times \Aut Q$. Having fixed a lifting $l: Q \rightarrow A$, we again see that $\phi \circ l \circ \phi^{-1}_{l}: Q \rightarrow A$ is another lifting of $\pi: A \rightarrow Q$ and by direct definition of a 2-cocycle defined by a lifting we see that $T^{(\phi|_{I},\phi_{l})}$ is the 2-cocycle defined by the lifting $\phi \circ l \circ \phi^{-1}_{l}$. Then by \cite{wires1} we have $T \sim T^{(\phi|_{I},\phi_{l})}$ which yields $W_{T}\circ \psi(\phi) = 0$; therefore, $\im \psi \leq \Aut I \times \Aut Q$.

We now assume $(\sigma,\kappa) \in \ker W_{T}$; equivalently, $[T] = [T^{(\sigma,\kappa)}]$ with $(\sigma,\kappa) \in \Aut I \times \Aut Q$. Then there exists a map $h: Q \rightarrow I$ such that $h(0) = 0$ and after pre-composing with $\kappa^{-1}$ we have the following:
\begin{enumerate}

	\item[(E1)] $T_{+}(\kappa(x),\kappa(y)) - \sigma \circ T_{+}(x,y) = h(\kappa(x)) + h(\kappa(y)) - h(\kappa(x) + \kappa(y))$;
	
	\item[(E2)] $T_{r}(\kappa(x)) - \sigma \circ T_{r}(x) = r \cdot h(\kappa(x)) - h(r \cdot \kappa(x))$ for each $r \in R$;
	
	\item[(E3)] for each $f \in F$, 
	\[
	T_{f}(\kappa(\vec{x})) - \sigma \circ T_{f}(\vec{x}) = \sum_{s \in [\ar f]^{\ast}} (-1)^{1+|s|} a(f,s)(\kappa(\vec{x}),h(\kappa(\vec{x}))) + (-1)^{1+n}f^{I}(h(\kappa(\vec{x}))) - h(f^{Q}(\kappa(\vec{x})));
	\]
	
	\item[(E4)] for each $f \in F$ and $s \in [\ar f]^{\ast}$,
\begin{align*}
	a(f,s)(\kappa(\vec{x}),\sigma(\vec{a})) - \sigma \circ a(f,s)(\vec{x},\vec{a}) &= \sum_{s \subseteq r \subseteq [\ar f]} (-1)^{1 + |r| - |s|} a(f,r)\left( \kappa(\vec{x}), h(\kappa(\vec{x})) \right)_{s}[ \sigma(\vec{a}) ] \\
	&\quad + (-1)^{1 + |\ar f| - |s|}f(h(\kappa(\vec{x})))_{s}[ \sigma(\vec{a}) ].
\end{align*}

\end{enumerate}
At this point, we explicitly use the isomorphism $A \approx I \rtimes_{T} Q$ from Wires \cite{multiext} given by $A \ni x \longmapsto \left\langle x - l \circ \pi(x), \pi(x) \right\rangle$. Define $l(\sigma,\kappa) : I \rtimes_{T} Q \rightarrow I \rtimes_{T} Q$ by the rule $\hat{l}(\sigma,\kappa)\left\langle a,x \right\rangle := \left\langle \sigma(a) - h(\kappa(x)), \kappa(x) \right\rangle$.

To show $\hat{l}(\sigma,\kappa)$ is a homomorphism, take a multilinear operation $f \in F$ with $n = \ar f$ and $\vec{a} \in I^{n}, \vec{x} \in Q^{n}$. Using (E3) and (E4), we calculate
\begin{align*}
F_{f} &\left( \hat{l}(\sigma,\kappa) \left\langle a_{1},x_{1} \right\rangle,\ldots,\hat{l}(\sigma,\kappa)\left\langle a_{n},x_{n} \right\rangle \right) \displaybreak[0]\\
&= F_{f} \left( \left\langle \sigma(a_{1}) - h(\kappa(x_{1})), \kappa(x_{1})  \right\rangle,\ldots,\left\langle \sigma(a_{n}) - h(\kappa(x_{n})), \kappa(x_{n})  \right\rangle \right)  \displaybreak[0]\\
&= \left\langle f^{I}\left(\sigma(\vec{a}) - h(\kappa(\vec{x})) \right) + \sum_{s \in [n]^{\ast}} a(f,s)\left(\kappa(\vec{x}), \sigma(\vec{a}) - h(\kappa(\vec{x}) \right) + T_{f}(\kappa(\vec{x})) , f^{Q}(\kappa(\vec{x})) \right\rangle \displaybreak[0]\\
&= \Bigg<  f^{I}(\sigma(\vec{a})) + \sum_{ \emptyset \neq t \subseteq [n]} (-1)^{|t|} f^{I} \left( \sigma(\vec{a}) \right)_{t}[h(\kappa(\vec{x}))] + \sum_{s \in [n]^{\ast}} \sum_{r \subseteq s} (-1)^{|r|} a(f,s)\left( \kappa(\vec{x}), \sigma(\vec{a}) \right)_{r}[h(\kappa(\vec{x}))] \\
&\quad + T_{f}(\kappa(\vec{x})), \kappa \left( f^{Q}(\vec{x}) \right) \Bigg> \displaybreak[0]\\
&= \Bigg< f^{I}(\sigma(\vec{a})) + \sum_{u \in [n]^{\ast}} (-1)^{n - |u|} f^{I} \left(h(\kappa(\vec{x})) \right)_{u}[\sigma(\vec{a})]  +  (-1)^{n} f^{I} \left( h(\kappa(\vec{x})) \right)  \displaybreak[0]\\
&\quad + \sum_{u \in [n]^{\ast}} \sum_{u \subseteq v \in [n]^{\ast}} (-1)^{|v| - |u|} a(f,v)(\kappa(\vec{x}),h(\kappa(\vec{x})))_{u}[\sigma(\vec{a})]  + T_{f}(\kappa(\vec{x})) , \kappa \left(f^{Q}(\vec{x}) \right)  \Bigg>  \displaybreak[0]\\ 
&= \left\langle  \sigma \circ f^{I}(\vec{a}) + \sum_{s \in [n]^{\ast}} \sigma \circ a(f,s)(\vec{x},\vec{a}) + \sigma \circ T_{f}(\vec{x}) - h \left( \kappa(f^{Q}(\vec{x})) \right), \kappa \left(f^{Q}(\vec{x}) \right) \right\rangle \displaybreak[0]\\
&= \hat{l}(\sigma,\kappa) \left\langle f^{I} \left( \vec{a} \right) + \sum_{s \in [n]^{\ast}} a(f,s)(\vec{x},\vec{a}) + T_{f}(\vec{x}), f^{Q}(\vec{x}) \right\rangle \displaybreak[0]\\
&= \hat{l}(\sigma,\kappa) \left( F_{f}\left( \left\langle a_{1},x_{1} \right\rangle, \ldots, \left\langle a_{1},x_{1} \right\rangle \right) \right).
\end{align*}
For the module operations, using (E1) and (E2) we have
\begin{align*}
r \cdot &\hat{l}(\sigma,\kappa)\left\langle a,x\right\rangle + \hat{l}(\sigma,\kappa)\left\langle b,y \right\rangle \displaybreak[0]\\
&= \left\langle r \cdot \sigma(a) - r \cdot h(\kappa(x)) + T_{r}(\kappa(x)) + \sigma(b) - h(\kappa(y)) + T_{+}(r \cdot \kappa(x),\kappa(y)), r \cdot \kappa(x) + \kappa(y) \right\rangle \displaybreak[0]\\
&= \left\langle \sigma(r \cdot a) + \sigma \circ T_{r}(x) - h(r \cdot \kappa(x)) + \sigma(b) + \sigma \cdot T_{+}(r \cdot x,y) + h(\kappa(r \cdot x)), \kappa(r \cdot x + y) \right\rangle \displaybreak[0]\\
&= \hat{l}(\sigma,\kappa) \left\langle r \cdot a + T_{r}(x) + b + T_{+}(r \cdot x, y) , r \cdot x + y \right\rangle \displaybreak[0]\\
&= \hat{l}(\sigma,\kappa)\left( r \cdot \left\langle a,x \right\rangle + \left\langle b,y \right\rangle \right).
\end{align*}
We have shown $\hat{l}(\sigma,\kappa)$ is a homomorphism. It is straightforward to see that $\hat{l}(\sigma,\kappa)$ is bijective such that $\hat{l}(\sigma,\kappa)(I \times 0) = I \times 0$; therefore, $\hat{l}(\sigma,\kappa) \in \Aut_{I \times 0} I \rtimes_{T} Q$.

The extension $p_{2}: I \rtimes_{T} Q \rightarrow Q$ given by second-projection has the lifting $r: Q \rightarrow I \rtimes_{T} Q$ given by $r(x) = \left\langle 0,x \right\rangle$. Then we easily see that $\hat{l}(\sigma,\kappa)_{r} = p_{2} \circ \hat{l}(\sigma,\kappa) \circ r = \kappa$. If we allow for the identification $I$ with $I \times 0$, then $\hat{l}(\sigma,\kappa)|_{I \times 0} = \sigma$; altogether, we see that $\psi \left( \hat{l}(\sigma,\kappa) \right) = (\sigma,\kappa)$ and so have demonstrated exactness at $\Aut I \times \Aut Q$ and finished the proof of the theorem.
\end{proof}

\begin{acknowledgments}
The research in this manuscript was supported by NSF China Grant \#12071374.
\end{acknowledgments}

\vspace{0.4cm}


\end{document}